\documentclass{amsart}
\usepackage[centertags]{amsmath}
\usepackage{amsfonts}
\usepackage{amssymb}
\usepackage{amsthm}
\allowdisplaybreaks
\usepackage[colorlinks]{hyperref}
\usepackage{tikz}
\usepackage{tikz-cd}
\usetikzlibrary{arrows.meta,calc,positioning,cd,arrows,mindmap}
\usepackage[normalem]{ulem}
\usepackage[shortlabels]{enumitem} 
\usepackage{marginnote}
\usepackage[all]{xy}
\usepackage{cancel}
\usepackage{bm}
\usepackage{tikz}
\usepackage{tikz-cd}
\usepackage{array,multirow,makecell}
\usepackage{caption}

\usepackage{xcolor} 

\newcommand{\fonction}[5]{\begin{array}{l rcl}
#1: & #2 & \longrightarrow & #3\\
      & #4 & \longmapsto & #5 \end{array}}

\newcommand{\N}{\mathbb{N}}

\newcommand{\C}{\mathbb{C}}
\newcommand{\D}{\mathbb{D}}

\newcommand{\G}{\mathcal{G}}
\newcommand{\Mol}{\mathcal{M}}

\newcommand{\Complex}{\mathbb C}

\newcommand\restr[2]{{
		\left.\kern-\nulldelimiterspace 
		#1 
		\littletaller 
		\right|_{#2} 
}}
\newcommand{\littletaller}{\mathchoice{\vphantom{\big|}}{}{}{}}

\newcommand{\linsp}{{\mathrm{span}}}

\newcommand{\Lip}{{\mathrm{Lip}}}
\newcommand{\HL}{{\mathrm{HL}}}

\newcommand{\F}{\mathcal{F}}              
\newcommand{\conv}{\mathop\mathrm{conv}}
\newcommand{\supp}{\mathop\mathrm{supp}}

\def\<{\langle}
\def\>{\rangle}

\renewcommand{\hat}{\widehat}

\theoremstyle{plain}
\newtheorem{thm}{Theorem}[section]
\newtheorem{theorem}[thm]{Theorem}
\newtheorem{cor}[thm]{Corollary}
\newtheorem{corollary}[thm]{Corollary}

\newtheorem{lemma}[thm]{Lemma}
\newtheorem{prop}[thm]{Proposition}
\newtheorem{proposition}[thm]{Proposition}

\newtheorem{maintheorem}{Theorem} 

\theoremstyle{definition}

\newtheorem{definition}[thm]{Definition}

\newtheorem{remark}[thm]{Remark}

\newtheorem{example}[thm]{Example}

\author[F. Bayart]{Frédéric Bayart}
\address[F. Bayart]{Laboratoire de Mathématiques Blaise Pascal UMR 6620 CNRS, Université Clermont Auvergne, Campus universitaire des Cézeaux, 3 place Vasarely, 63178 Aubière Cedex, France.}
\email{frederic.bayart@uca.fr}

\author[M. Lemay]{Mathis Lemay}
\address[M. Lemay]{Univ Gustave Eiffel, Univ Paris Est Creteil, CNRS, LAMA UMR8050, F-77447 Marne-la-Vallée, France}
\email{mathis.lemay@univ-eiffel.fr}

 \author[C. Petitjean]{Colin Petitjean}
\address[C. Petitjean]{Univ Gustave Eiffel, Univ Paris Est Creteil, CNRS, LAMA UMR8050, F-77447 Marne-la-Vallée, France}
\email{colin.petitjean@univ-eiffel.fr}

\begin{document}
	\title[Compactness and weak compactness of linearized HL maps]{Compactness and weak compactness of linearized holomorphic-Lipschitz maps}
	


\begin{abstract}
Any holomorphic Lipschitz map $f$, defined on the open unit ball $B_X$ of a complex Banach space $X$ and taking values in a complex Banach space $Y$, can be uniquely and canonically associated with a bounded linear operator $\widehat f : \mathcal G_0(B_X) \longrightarrow \mathcal G_0(Y)$
between the corresponding holomorphic Lipschitz-free spaces. The aim of this paper is to establish necessary and sufficient conditions for the compactness and weak compactness of such operators.

In particular, we obtain a general characterisation of the compactness of $\widehat f$ in terms of metric properties of its symbol $f$. We then focus on the case where the target space is a subspace of $\ell_p$, $1\leq p<\infty$. When the target space is $\ell_1$, we show that compactness and weak compactness of $\widehat f$ are both equivalent to $\widehat f$ not fixing a copy of $\ell_1$. Finally, when the target space is $\ell_p$ with $1<p<\infty$, we obtain a metric characterisation of the weak compactness of $\widehat f$. As a consequence, we show that compactness and weak compactness need not coincide, even when the target space is $\ell_p$.
\end{abstract}

	\maketitle

\section{Introduction}

Let $X$ and $Y$ be complex Banach spaces. We write $\mathcal L(X,Y)$ for the Banach space of bounded linear operators from $X$ to $Y$ and $X^*=\mathcal L(X,\Complex)$ for the (continuous) dual of $X$. We denote by $B_X$ the open unit ball of $X$ and by $S_X$ its unit sphere.

A mapping $f\colon B_X\to Y$ is said to be \emph{holomorphic} (or Fr\'echet holomorphic) if for every $x\in B_X$ there exists an operator $df(x)\in\mathcal L(X,Y)$ such that
$$
\lim_{z\to 0}\ \sup_{h\in S_X}\left\|\frac{f(x+zh)-f(x)}{z}-df(x)(h)\right\|_Y =0.
$$
The operator $df(x)$ is unique and is called the (Fr\'echet) differential of $f$ at $x$. We refer to \cite{muj} for background on holomorphy and differential calculus in Banach spaces.

We say that $f$ is \emph{Lipschitz} on $B_X$ if there exists $C>0$ such that
$$
\|f(x)-f(y)\|_Y\le C\|x-y\|_X\qquad (x,y\in B_X),
$$
and we write $\|f\|_L$ for the best such constant. We set
$$
\HL_0(B_X,Y)=\{f\colon B_X\to Y:\ f \text{ is holomorphic and Lipschitz, and } f(0)=0\},
$$
endowed with the norm $\|f\|_L$. This is a Banach space; when $Y=\Complex$ we simply write $\HL_0(B_X)=\HL_0(B_X,\Complex)$. It was shown in \cite{aron} that for $f\in \HL_0(B_X,Y)$ one has the identity
$$
\|f\|_L=\sup_{x\in B_X}\|df(x)\|.
$$

Following \cite{aron}, we consider the evaluation map
$$
\delta\colon B_X\longrightarrow \HL_0(B_X)^*,\qquad \delta(x)(g)=g(x)\ \ (g\in \HL_0(B_X)).
$$
This map is a holomorphic isometric embedding, and the set $\{\delta(x):x\in B_X\}$ is linearly independent. The \emph{holomorphic--Lipschitz free space} over $B_X$ is defined by
$$
\mathcal G_0(B_X)=\overline{\linsp}\{\delta(x):x\in B_X\}\subset \HL_0(B_X)^*,
$$
where the closure is taken in the norm topology of $\HL_0(B_X)^*$.
The fundamental universal property of $\mathcal G_0(B_X)$ (see \cite{aron}) states that for every $f\in \HL_0(B_X,Y)$ there exists a unique bounded linear operator $\overline f\colon \mathcal G_0(B_X)\to Y$ satisfying $\overline f(\delta(x))=f(x)$ for all $x\in B_X$, and moreover $\|\overline f\|=\|f\|_L$.

In particular, if $f\in \HL_0(B_X,Y)$ satisfies $f(B_X)\subset B_Y$, there is a unique bounded operator
$$
\widehat f\colon \mathcal G_0(B_X)\longrightarrow \mathcal G_0(B_Y)
\qquad\text{such that}\qquad
\delta\circ f=\widehat f\circ \delta,
$$
and $\|\widehat f\|=\|f\|_L$. We call $\widehat f$ the \emph{holomorphic--Lipschitz operator} associated with~$f$.

As in \cite{ACp54}, one checks that the adjoint of $\widehat f$ is the composition operator
$$
C_f\colon \HL_0(B_Y)\longrightarrow \HL_0(B_X),\qquad C_f(g)=g\circ f.
$$
By Schauder's theorem \cite[Theorem~3.4.15]{Megg} and Gantmacher's theorem \cite[Theorem~3.5.13]{Megg},
compactness (respectively, weak compactness) of $\widehat f$ is equivalent to compactness (respectively, weak compactness) of $C_f$.

The goal of this paper is to provide usable sufficient and necessary conditions for $\widehat f$ (equivalently, $C_f$) to be compact or weakly compact. We pay particular attention to the range $Y=\ell_p$ for $1\le p<\infty$ (with conjugate exponent $q$).

While this work was in preparation, the related paper \cite{LuisCarlos} appeared, with a nontrivial overlap with some of the questions considered here. Although both approaches rely on the linearization through holomorphic-Lipschitz free spaces, they are rather different and complementary. Our approach is more direct and largely self-contained, and focuses primarily on the linearized operator $\widehat f$, whereas \cite{LuisCarlos} studies the adjoint composition operator $C_f$ and draws more extensively on the theory of composition operators on spaces of holomorphic functions, notably through factorization and commutative-diagram arguments. This difference in viewpoint also allows us to obtain several results, such as e.g. Theorems~\ref{thmB}, \ref{thmD} and \ref{thmE} below,  which are not covered by \cite{LuisCarlos}. We encourage the interested reader to consult \cite{LuisCarlos} in parallel for a broader perspective on composition operators on spaces of holomorphic Lipschitz functions.

\medskip
\noindent\textbf{Overview of the main results.} 

We first establish several preliminary results which will be used throughout the paper. In particular, we study different notions expressing that the image $f(B_X)$ remains strictly inside $B_Y$, and we characterize compactness and weak compactness of $\widehat{f}$ in terms of the image under $\widehat{f}$ of the set of molecules of $\G_0(B_X)$. We will also repeatedly use the following notation:
$$
L_f:=\left\{\frac{f(x)-f(y)}{\|x-y\|} : x \neq y \in B_X\right\}.
$$
The set $L_f$ may be viewed as the Lipschitz image of $f$ and will play an important role in our characterizations of compactness and weak compactness.
\medskip

We then turn to the case where the range is an $\ell_p$-space, beginning with $\ell_1$. One of the main tools in this part is the Bessaga--Pe{\l}czy\'nski gliding hump principle (see e.g. \cite[Proposition~1.3.10]{AK}). Combined with a detailed analysis of sequences of molecules, it allows us to prove that, for $\ell_1$-valued symbols, weak compactness and compactness of $\widehat{f}$ coincide. In fact, these properties are also equivalent to several standard operator-theoretic properties.

\begin{maintheorem}\label{thmB}
Let $f \in \HL_0(B_X,\ell_1)$ such that $f(B_X) \subset B_{\ell_1}$. The following assertions are equivalent:
\begin{enumerate}
    \item[(i)] $\widehat{f} : \G_0(B_X) \longrightarrow \G_0(B_{\ell_1})$ is compact;
    \item[(ii)] $\widehat{f} : \G_0(B_X) \longrightarrow \G_0(B_{\ell_1})$ is weakly compact;
    \item[(iii)] $\widehat{f} : \G_0(B_X) \longrightarrow \G_0(B_{\ell_1})$ is super strictly singular;
    \item[(iv)] $\widehat{f} : \G_0(B_X) \longrightarrow \G_0(B_{\ell_1})$ is strictly singular;
    \item[(v)] $\widehat{f}$ does not fix any copy of $\ell_1$.
\end{enumerate}
\end{maintheorem}

As a first consequence of Theorem~\ref{thmB}, we obtain the same equivalence between compactness and weak compactness whenever the range of $f$ is finite dimensional.

\begin{maintheorem}\label{thmC}
Let $f \in \HL_0(B_X,Y)$ satisfy $f(B_X) \subset B_Y$ and $\dim(Y)<\infty$. Then
$$
\widehat{f} \text{ is compact } \iff \widehat{f} \text{ is weakly compact.}
$$
\end{maintheorem}
\medskip

We next consider arbitrary Banach spaces $X$ and $Y$. Our first general result gives a metric characterization of compactness of the holomorphic--Lipschitz operator $\widehat{f}$, expressed entirely in terms of the symbol $f$.

\begin{maintheorem}\label{thmD}
Let $f \in \HL_0(B_X,Y)$ such that $f(B_X) \subset B_Y$. Then
$$
\widehat{f} : \G_0(B_X) \longrightarrow \G_0(B_Y)
\text{ is compact }
\iff
\overline{L_f} \text{ is compact and } \|f\|_{\infty}<1.
$$
\end{maintheorem}

The proof combines a sufficient condition for compactness with the necessary conditions established beforehand. In particular, compactness of $\widehat{f}$ forces $L_f$ to be relatively compact and, together with the compactness of $f(B_X)$ which follows from this condition, forces the image of $f$ to remain uniformly away from the boundary of $B_Y$.

We then study weak compactness. We first prove a general sufficient condition when the target space $Y$ is reflexive: if $\overline{f(B_X)}$ is a compact subset of $B_Y$, then $\widehat{f}$ is weakly compact. For targets $\ell_p$, $1<p<\infty$, this sufficient condition turns out to be necessary as well, and yields the following metric characterization.

\begin{maintheorem}\label{thmE}
Let $1<p<\infty$ and let $f \in \HL_0(B_X,\ell_p)$ such that $f(B_X) \subset B_{\ell_p}$. Then
$\widehat{f} : \mathcal{G}_0(B_X) \longrightarrow \mathcal{G}_0(B_{\ell_p})$ verifies: $$
\widehat{f} \text{ is weakly compact }
\iff
\overline{f(B_X)} \text{ is compact and } \|f\|_{\infty}<1.
$$
\end{maintheorem}

When $p=1$, one implication remains true:
$$
\widehat{f} \text{ is weakly compact }
\implies
\overline{f(B_X)} \text{ is compact and } \|f\|_{\infty}<1,
$$
whereas the converse fails.

Comparing Theorems~\ref{thmD} and~\ref{thmE}, we see that, for $1<p<\infty$, compactness is governed by the compactness of $\overline{L_f}$, whereas weak compactness only requires the compactness of $\overline{f(B_X)}$, together in both cases with the condition $\|f\|_{\infty}<1$. Since compactness of $\overline{L_f}$ is strictly stronger than compactness of $\overline{f(B_X)}$, this leads to an explicit map $f \in \HL_0(B_{\ell_p},\ell_p)$
such that $f(B_{\ell_p}) \subset B_{\ell_p}$ and $\widehat{f}$ is weakly compact but not compact. 

This emphasizes a difference with the purely Lipschitz setting. Indeed, in \cite{ACp51}, the authors obtained a metric characterization of compactness for the linearization $\widehat{f} : \F(M) \to \F(N)$
of a Lipschitz map between pointed metric spaces. In that setting, compactness and weak compactness of $\widehat{f}$ always coincide. This is no longer the case for holomorphic--Lipschitz maps, as the results below will show. This shows that holomorphic--Lipschitz operators form a substantially richer and more diverse class than their counterparts in the classical Lipschitz-free space setting.
\medskip

Finally, we apply these characterizations to some particular situations. For linear symbols with values in $\ell_p$, compactness and weak compactness of $\widehat{f}$ coincide. We also investigate the case where the domain $X$ is finite dimensional. When $Y$ is reflexive, weak compactness of $\widehat{f}$ admits a particularly simple characterization in terms of $\|f\|_{\infty}<1$. However, finite dimensionality of the domain does not force compactness and weak compactness to coincide, even for reflexive targets. This contrasts with Theorem~\ref{thmC} and illustrates the different roles played by the domain and the range in the study of holomorphic--Lipschitz operators.

\section{Preliminary results}

\subsection{Various properties of inclusion of images}

We start by studying the relationships between various properties of strict inclusion of $f(B_X)$ into $B_Y$ for a map $f\in\HL_0(B_X,Y)$. We shall need the following lemma.

\begin{lemma}\label{contp4}
Let $g \in HL_0(B_X,\mathbb{C})$ such that $g(B_X) \subset \mathbb{D}$. If $1 \in \overline{g(B_X)}$, then there exists a sequence $(z_n) \subset B_X$ such that
$g(z_n) \underset{n \rightarrow +\infty}{\longrightarrow} 1 \text{ and } \underset{n \rightarrow +\infty}{\liminf} \|dg(z_n)\| \geq 1$.
\end{lemma}

\begin{proof}
    Since $1 \in \overline{g(B_X)}$, let $(x_n) \subset B_X$ such that $g(x_n) \underset{n \rightarrow +\infty}{\longrightarrow} 1$ and write $x_n=r_nu_n$ where $r_n=\|x_n\|$ and $u_n=\frac{x_n}{\|x_n\|} \in S_X$. Applying Schwarz Lemma yields $|g(x_n)| \leq \|x_n\|=r_n$, therefore $r_n \underset{n \rightarrow +\infty}{\longrightarrow} 1$. 

    Consider now the real-valued differentiable function:

    $$\varphi_n : t \in [0,r_n] \longmapsto \mathrm{Re}(g(t(u_n)).$$
    Define $\eta_n=r_n-\mathrm{Re}(g(x_n))$. We have that $\eta_n \geq 0$, thus one can define $\varepsilon_n=\sqrt{\eta_n}+\frac{1}{n}$. Then:
    $$\eta_n \underset{n \rightarrow +\infty}{\longrightarrow} 0, \quad \varepsilon_n \underset{n \rightarrow +\infty}{\longrightarrow} 0 \text{ and } \frac{\eta_n}{\varepsilon_n} \leq \frac{\eta_n}{\sqrt{\eta_n}} \underset{n \rightarrow +\infty}{\longrightarrow} 0.$$
    Again, by Schwarz's lemma:
    $$\varphi_n(r_n-\varepsilon_n) \leq |g((r_n-\varepsilon_n)u_n)| \leq |r_n-\varepsilon_n|$$
    Note that since $r_n \underset{n \rightarrow +\infty}{\longrightarrow} 1$ and $\varepsilon_n \underset{n \rightarrow +\infty}{\longrightarrow} 0$, for $n$ large enough, we have that $r_n-\varepsilon_n \geq 0$. In the sequel, we assume that $r_n-\varepsilon_n \geq 0$.
    
    By the mean value theorem, there exists $t_n \in ]r_n-\varepsilon_n,r_n[$ such that:
    $$\varphi_n'(t_n)=\frac{\varphi_n(r_n)-\varphi_n(r_n-\varepsilon_n)}{\varepsilon_n},$$
    hence, 
    $$\varphi_n'(t_n) \geq \frac{r_n-\eta_n-(r_n-\varepsilon_n)}{\varepsilon_n} \geq 1-\frac{\eta_n}{\varepsilon_n} \underset{n \rightarrow +\infty}{\longrightarrow} 1.$$
    To conclude, let $z_n=t_nu_n$. Since we have that $\varphi_n'(t_n)=\mathrm{Re}(dg(z_n)(u_n))$, we obtain:
    $$\|dg(z_n)\| \geq |dg(z_n)(u_n)| \geq 1-\frac{\eta_n}{\varepsilon_n}$$
    Thus, $\underset{n \rightarrow +\infty}{\liminf} \|dg(z_n)\| \geq 1$, and:
    $$|g(z_n)-g(x_n)| \leq \|g\|_L\|z_n-x_n\| \leq \|g\|_L(r_n-t_n)$$
    So $g(z_n) \underset{n \rightarrow +\infty}{\longrightarrow} 1$ and $(z_n)$ is the sequence in $B_X$ we were looking for.
\end{proof}

\begin{proposition}[Relations between various ``strict interior'' properties]\label{prop-images-relation}
	Let $X,Y$ be complex Banach spaces and let $f\in \HL_0(B_X,Y)$ satisfy $f(B_X)\subset B_Y$.
	Consider the following properties:
	\begin{enumerate}[(P1)]
		\item \label{p1} $\displaystyle \|f\|_\infty:=\sup_{x\in B_X}\|f(x)\|<1$;
        \item \label{p2} For all $y^*\in S_{Y^*},$ for all $(z_j)\subset B_X$ such that $y^*(f(z_j))\to 1$, then $\|y^*\circ df(z_j)\|\to 0$;
        \item \label{p3} $\overline{f(B_X)}\subset B_Y$ (closure in the norm topology of $Y$);
		\item \label{p4} $f(S_X)\cap S_Y=\varnothing$;        
		\item \label{p5} $\|f(x)\|<\|x\|$ for every $x\in \overline{B_X}\setminus\{0\}$.
	\end{enumerate}
	Then the following implications always hold:
    \[
    \textnormal{\ref{p1}}\
    \Longrightarrow \textnormal{\ref{p2}}\ 
    \Longrightarrow
    \textnormal{\ref{p3}}\ 
    \Longrightarrow
    \textnormal{\ref{p4}}\
    \Longleftrightarrow\ 
    \textnormal{\ref{p5}}
    \]
    If $\dim(X)<\infty,$ all these properties are equivalent, and if $\overline{f(B_X)}$ is norm compact in $Y$ (for instance if 
    $Y$	is finite-dimensional), then $\textnormal{\ref{p1}}\ \Longleftrightarrow\ \textnormal{\ref{p3}}$.\\
	In general, none of the converses of the above implications holds, even for linear maps.
\end{proposition}

\begin{proof}
\noindent\emph{ $\textnormal{\ref{p1}}\implies\textnormal{\ref{p2}}$}
Since $\overline{f(B_X)}\subset r\,\overline{B_Y}$ for some $r\in(0,1),$
there is no sequence $(z_j)\subset B_X$ and no $y^*\in S_{Y^*}$ such that $y^*(f(z_j))\to 1.$

\smallskip 
\noindent\emph{ $\textnormal{\ref{p2}}\implies\textnormal{\ref{p3}}$} Assume that $\overline{f(B_X )} \cap S_Y \neq \emptyset$. Let $y \in \overline{f(B_X)} \cap S_Y$. There exists $(x_n) \subset B_X$ such that $f(x_n) \underset{n \rightarrow +\infty}{\longrightarrow} y$. Using Hahn-Banach theorem, consider $y^* \in Y^*$ such that $\langle y^*,y\rangle=1$. Define $g=y^* \circ f \in HL_0(B_X,\mathbb{C})$. Clearly, we have that $g(B_X) \subset \mathbb{D}$, and $1 \in \overline{g(B_X)}$ since $g(x_n) \underset{n \rightarrow +\infty}{\longrightarrow} 1$. By Lemma \ref{contp4}, one has the existence of a sequence $(z_n) \subset B_X$ such that $g(z_n) \underset{n \rightarrow +\infty}{\longrightarrow} 1$ and $\underset{n \rightarrow +\infty}{\liminf} \|dg(z_n)\| \geq 1$. Recalling that $g=y^* \circ f$, it contradicts $\textnormal{\ref{p2}}$. 

\smallskip

\noindent\emph{ $\textnormal{\ref{p3}}\implies\textnormal{\ref{p4}}$}.
    Clearly \ref{p3} implies
	$f(S_X)\subset B_Y$, and therefore $f(S_X) \cap S_Y=\emptyset$, i.e.\ \ref{p4}.
	\smallskip
    
	\noindent\emph{ $\textnormal{\ref{p4}}\implies\textnormal{\ref{p5}}$ (strict Schwarz lemma).}
	Fix $x\in \overline{B_X}\setminus\{0\}$ and set $u:=\frac{1}{\|x\|}x\in S_X$, $r:=\|x\|\in(0,1]$.
	Choose $y^*\in S_{Y^*}$ such that $|y^*(f(x))|=\|f(x)\|$.
	Consider the holomorphic function
	\[
	\psi(\xi):=y^*(f(\xi u)),\qquad \xi\in\D.
	\]
	Then $\psi(0)=0$ and $|\psi(\xi)|\le |\xi|$ for $\xi\in\D$ by Schwarz' lemma, hence
	\[
	\|f(x)\|=|\psi(r)|\le r=\|x\|.
	\]
	Assume that equality holds. Then Schwarz' lemma forces the existence of a $\theta \in \mathbb{R}$ such that $\psi(\xi)=e^{i\theta}\xi$, and in particular
	$|\psi(1)|=1$. This implies $\|f(u)\|\ge |y^*(f(u))|=|\psi(1)|=1$, hence $\|f(u)\|=1$, contradicting \ref{p4}.
	Therefore $\|f(x)\|<\|x\|$, proving \ref{p5}.

    \smallskip
	\noindent\emph{ $\textnormal{\ref{p5}}\implies\textnormal{\ref{p4}}$.}
	If $x\in S_X$, then \ref{p5} gives $\|f(x)\|<\|x\|=1$, hence $f(S_X)\cap S_Y=\varnothing$.
	
	\smallskip
	\noindent\emph{$\textnormal{\ref{p3}}\iff\textnormal{\ref{p1}}$ whenever $\overline{f(B_X)}$ is compact.}
	We already proved \ref{p1}$\implies$\ref{p3}. Conversely, assume that $\overline{f(B_X)}$ is norm compact in $Y$
	and satisfies \ref{p3}. Since the norm is continuous, it attains its maximum on the compact set $\overline{f(B_X)}$:
	\[
	M:=\max_{y\in \overline{f(B_X)}}\|y\|<1,
	\]
	because $\overline{f(B_X)}\subset B_Y$ and $B_Y$ is open. Hence $\sup\limits_{x\in B_X}\|f(x)\|\le M<1$, that is \ref{p1}.

    \smallskip
    \noindent\emph{$\textnormal{\ref{p1}}
    \iff\textnormal{\ref{p4}}$ whenever  $\dim(X)<\infty$.} This follows easily from the compactness of $\overline{B_X}.$

	\smallskip
	\noindent\emph{Failure of converses in general.}
	We give a simple linear example showing that
	\textnormal{\ref{p3}} does not imply \textnormal{\ref{p1}} 
    for infinite-dimensional Banach spaces.
	Let $X=Y=\ell_2$ and define the diagonal operator $T\colon \ell_2\to \ell_2$ by
	\[
	T(x_1,x_2,x_3,\dots)=\Big(\tfrac12 x_1,\ \tfrac23 x_2,\ \tfrac34 x_3,\dots,\tfrac{n}{n+1}x_n,\dots\Big).
	\]
	Then $\|T\|=\sup_{n}\frac{n}{n+1}=1$, so \textnormal{\ref{p1}} fails.
	However, for every $x\neq 0$,
	\[
	\|Tx\|_2^2=\sum_{n\ge 1}\Big(\tfrac{n}{n+1}\Big)^2|x_n|^2<\sum_{n\ge 1}|x_n|^2=\|x\|_2^2,
	\]
	so \textnormal{\ref{p5}} (hence \textnormal{\ref{p4}}) holds. Moreover, since $\inf_n \frac{n}{n+1}=\frac12>0$,
	the operator $T$ is an isomorphism onto its range; in particular, $T(\overline{B_{\ell_2}})$ is closed. As it is also
	contained in $B_{\ell_2}$, we get $\overline{T(B_{\ell_2})}=T(\overline{B_{\ell_2}})\subset B_{\ell_2}$, so
	\textnormal{\ref{p3}} holds as well. This shows that \textnormal{\ref{p3}} does not imply
	\textnormal{\ref{p1}} in general.

    We now show that \textnormal{\ref{p4}} does not imply \textnormal{\ref{p3}}. Indeed, consider:
    $$\fonction{f}{\overline{B_{\ell_1}}}{\mathbb{C}}{(x_n)}{\sum\limits_{n=1}^{+\infty}\left(1-\frac{1}{n}\right)x_n.}$$
Since $f$ is linear, $f \in HL_0(B_{\ell_1},\mathbb{C})$. If $(e_n)$ denote the canonical basis of $\ell_1$, we have that 
$$
f\Big(\big(1-\frac 1 n\big)e_n\Big)=\Big(1-\frac{1}{n}\Big)^2 \underset{n \to \infty}{\longrightarrow} 1 \text{ so } 1 \in \overline{f(B_{\ell_1})}.
$$
In addition, let $x=(x_n) \in S_{\ell_1}$. Then:
$$|f(x_n)|=\left|\sum_{n=1}^{+\infty} \left(1-\frac{1}{n}\right)x_n\right| \leq \sum_{n=1}^{+\infty} \left(1-\frac{1}{n}\right)|x_n|<\|x\|_1=1.$$
Hence, $f(S_{\ell_1}) \cap S_{\mathbb{C}}=\emptyset$ but $\overline{f(B_{\ell_1})} \cap S_{\mathbb{C}}\neq \emptyset$. This shows that $f$ satisfies \textnormal{\ref{p4}} but not \textnormal{\ref{p3}}. The same example shows that \textnormal{\ref{p4}} does not imply \textnormal{\ref{p2}}. Indeed, if we take $z_j=\left(1-\frac 1j\right)e_j$ for all $j\geq 1,$ then $f(z_j)\to 1$ and 
$df(z_j)$ does not converge to $0.$

A similar example shows that \textnormal{\ref{p2}} does not imply \ref{p1}. 
Indeed, consider
    $$\fonction{f}{\overline{B_{c_0}}}{c_0}{(x_n)}{\left(\left(1-\frac{1}{n+1}\right)x_n\right).}$$
Then $f\in HL_0(B_{c_0},c_0)$ with $f(B_{c_0})\subset B_{c_0}$ and $\|f\|_\infty=1.$
For all $y^*\in S_{\ell_1}$  and all $z\in B_{c_0}$, 
$$|y^*(f(z))|\leq \sum_{n\geq 1}|y^*_n|\left(1-\frac 1{n+1}\right) <1.$$
Therefore, there does not exist a sequence $(z_j)\subset B_{c_0}$ such that $y^*(f(z_j))\to 1$
and $f$ satisfies \ref{p2}.

Finally, let us prove that \textnormal{\ref{p3}} does not imply \textnormal{\ref{p2}}. Indeed, consider 
  $$\fonction{f}{\overline{B_{\ell_1}}}{\ell_1}{(x_n)_n}{\left(\left(1-\frac{1}{n+1}\right)x_n\right)_{n}.}$$
As above, $f$ satisfies \textnormal{\ref{p3}}. It does not satisfy \textnormal{\ref{p2}}. 
Indeed, take $y^*=(1,\dots)$ and $z_j=\left(1-\frac1{j+1}\right)e_j$ so that 
\[y^*\circ f(z_j)=\left(1-\frac 1{j+1}\right)^2\to 1
\quad \text{ and } \quad
\left\|y^*\circ df(z_j)\right\|\to 1.\]
\end{proof}

\begin{remark}\label{remp4p1}
If we assume that $\overline{f(B_X)}$ is compact and $f$ is linear, then $\textnormal{\ref{p2}}\implies\textnormal{\ref{p1}}$. Indeed, assume that $\|f\|_\infty=1.$ By compactness of $\overline{f(B_X)}$, there exists $y\in S_Y$ and $(x_j)\subset B_X$ such that $f(x_j)\to y.$ Let $y^*\in S_{X^*}$ such that $y^*(y)=1$. Then $y^*(f(x_j))\to 1$. Since $f$ is linear, for all $j,$ we know that $y^*(f(x_j))=y^*(df(x_j)(x_j))$. Hence \ref{p2} is not satisfied.
\end{remark}

\subsection{Molecules and compactness}

We denote by
$$
\mathcal M:=\left\{\frac{\delta(x)-\delta(y)}{\|x-y\|} : x,y\in B_X,\ x\neq y\right\}
$$
the set of \emph{elementary molecules} in $\G_0(B_X)$. Molecules play a central role in the geometry of
Lipschitz-free spaces. For instance, it is shown in \cite{ACp54} that in the free space $\mathcal F(M)$ the set of
molecules (and, more generally, the set of finitely supported elements with uniformly bounded support size) enjoys the
Schur property: weak convergence of sequences implies norm convergence on that set. We emphasize that the situation is
drastically different in $\G_0(B_X)$.

\begin{example}\label{ex:weak-not-norm-molecule}
	Let $1<p<\infty$ and let $(t_n)$ be a decreasing sequence in $(0,1]$ such that $t_n\to 0$. Denote by $(e_n)$ the
	canonical basis of $\ell_p$. For each $n\in\N$, set
	$$
	v_n:=\frac{\delta(t_n e_n)-\delta(0)}{t_n}\in \G_0(B_{\ell_p}).
	$$
	Then $(v_n)$ does not converge to $0$ in norm, but $v_n \overset{w}{\longrightarrow} 0$ in $\G_0(B_{\ell_p})$.
	
	\smallskip
	Indeed, since $\|\delta(u)-\delta(0)\|=\|u\|_p$ in $\G_0(B_{\ell_p})$, we have $\|v_n\|=1$ for all $n$, so there is no
	norm convergence to $0$. On the other hand, fix $f\in \HL_0(B_{\ell_p})$. Using the first-order Taylor expansion at $0$,
	$$
	f(t_n e_n)=df(0)(t_n e_n)+r(t_n e_n),
	\qquad\text{with}\qquad
	\frac{r(t_n e_n)}{t_n}\xrightarrow[n\to\infty]{}0.
	$$
	Therefore
	$$
	\langle f,v_n\rangle=\frac{f(t_n e_n)-f(0)}{t_n}
	=\langle df(0),e_n\rangle+\frac{r(t_n e_n)}{t_n}.
	$$
	Since $df(0)\in(\ell_p)^*=\ell_q$ and $e_n\to 0$ weakly in $\ell_p$, we have $\langle df(0),e_n\rangle\to 0$, and
	the remainder term also tends to $0$. Hence $\langle f,v_n\rangle\to 0$ for every $f\in \HL_0(B_{\ell_p})$, that is,
	$v_n\overset{w}{\longrightarrow} 0$ in $\G_0(B_{\ell_p})$.
\end{example}

Despite this pathology, (weak) compactness of operators $\widehat f$ can still be characterized in terms of the image of
elementary molecules, as a straightforward adaptation of the corresponding result for free linearizations of Lipschitz
maps.

\begin{prop}\label{mol}
Let $f \in \HL_0(B_X,Y)$ such that $f(B_X) \subset B_Y$. Then, the operator $\widehat{f} : \G_0(B_X) \longrightarrow \G_0(B_Y)$ is (weakly) compact if and only if the set: 
$$\widehat{f}(\mathcal{M})=\left\{\frac{\delta(f(x))-\delta(f(y))}{\|x-y\|} : x \neq y \in B_X\right\}$$
is relatively (weakly) compact.
\end{prop}

\begin{proof}
The proof is identical to that of Proposition 2.1 of \cite{ACp51}, since $\overline{B_{\G_0(B_X)}}$ is the closed absolute convex hull of $\mathcal{M}$ (see \cite[Proposition 2.8]{aron}).
\end{proof} 

The following two lemmas will be used repeatedly in the sequel. 

\begin{lemma}\label{Bayart}
Let $X$ be a complex Banach space. Let $(a_n)$ and $(b_n)$ be sequences in $B_X$ converging in norm to the same point $a\in B_X$, and let $(d_n)$ be a sequence of positive numbers with $d_n \to 0$. Set
$$
w_n:=\frac{a_n-b_n}{d_n}\in X
\qquad\text{and}\qquad
m_n:=\frac{\delta(a_n)-\delta(b_n)}{d_n}\in \mathcal G_0(B_X).
$$
If $w_n \overset{w}{\longrightarrow} w$ in $X$, then $m_n \overset{w}{\longrightarrow} e_{a,w}$ in $\mathcal G_0(B_X)$, where
$$
e_{a,w}(f)=df(a)(w)\qquad (f\in \HL_0(B_X)).
$$
Moreover, if $w_n \longrightarrow w$ in norm, then $m_n \longrightarrow e_{a,w}$ in norm in $\mathcal G_0(B_X)$.
\end{lemma}

\begin{proof}
Fix $h\in \HL_0(B_X)$. Recall that $h$ is Fr\'echet differentiable on $B_X$ and that the map
$z\mapsto dh(z)\in X^*$ is continuous. Moreover, $\|dh(z)\|\le \|h\|_L$ for all $z\in B_X$.

\medskip
\noindent\emph{Weak convergence.}
For each $n$, consider the segment $\gamma_n\colon[0,1]\to B_X$ given by
$\gamma_n(t)=b_n+t(a_n-b_n)$. By the (Fr\'echet) fundamental theorem of calculus along segments,
$$
h(a_n)-h(b_n)=\int_0^1 dh(\gamma_n(t))(a_n-b_n)\,dt.
$$
Dividing by $d_n$ and using $a_n-b_n=d_n w_n$, we get
$$
\langle m_n,h\rangle=\frac{h(a_n)-h(b_n)}{d_n}=\int_0^1 dh(\gamma_n(t))(w_n)\,dt.
$$
Hence
\begin{align*}
\big|\langle h, m_n-e_{a,w}\rangle\big|
&=\left|\int_0^1 \big(dh(\gamma_n(t))(w_n)-dh(a)(w)\big)\,dt\right|\\
&\le \int_0^1 \big|(dh(\gamma_n(t))-dh(a))(w_n)\big|\,dt \;+\; \big|dh(a)(w_n-w)\big|.
\end{align*}
Since $w_n\to w$ weakly in $X$, the second term tends to $0$.
Moreover, $\gamma_n(t)\to a$ uniformly in $t\in[0,1]$ because $a_n\to a$ and $b_n\to a$ in norm, hence
$$
\sup_{t\in[0,1]}\|dh(\gamma_n(t))-dh(a)\|_{X^*}\xrightarrow[n\to\infty]{}0.
$$
Finally, $(w_n)$ is bounded (because it is weakly convergent), so the first term also tends to $0.
$ This proves that $m_n \overset{w}{\longrightarrow} e_{a,w}$ in $\G_0(B_X)$.

\medskip
\noindent\emph{Norm convergence.}
Assume now that $w_n\to w$ in norm.
Let
$$K:=\overline{\conv}\big({\{a_n:n\in\N\}\cup\{b_n:n\in\N\}}\cup\{a\}\big).$$

Then $K$ is compact and $K\subset B_X$.
Hence there exists $r>0$ such that
\begin{equation*}
K+rB_X\subset B_X.
\end{equation*}
Fix $h\in \HL_0(B_X)$ with $\|h\|_L\le 1$. Since $h(0)=0$ and $h$ is $1$-Lipschitz on $B_X$, we have
$\sup\limits_{z\in B_X}|h(z)|\le 1$.
By a standard Cauchy estimate on the balls $z+rB_X\subset B_X$ (for $z\in K$), there exists a constant
$M>0$ (depending only on $r$) such that
\begin{equation}\label{eq:bayart-D2}
\sup_{z\in K}\|d^2h(z)\|\le M \qquad\text{for all }h\in \HL_0(B_X)\text{ with }\|h\|_L\le 1.
\end{equation}
In particular, $dh$ is $M$-Lipschitz on $K$, i.e.
\begin{equation}\label{eq:bayart-dh-lip}
\|dh(u)-dh(v)\|\le M\|u-v\|\qquad (u,v\in K,\ \|h\|_L\le 1).
\end{equation}

Write the second order Taylor formula with integral remainder at $b_n$:
$$
h(a_n)-h(b_n)=dh(b_n)(a_n-b_n)+\int_0^1(1-t)\,d^2h\big(b_n+t(a_n-b_n)\big)[a_n-b_n,a_n-b_n]\,dt.
$$
Dividing by $d_n$ and using $a_n-b_n=d_n w_n$ yields
$$
\langle m_n,h\rangle
=dh(b_n)(w_n)+R_n(h),
$$
where
$$
R_n(h):=\frac{1}{d_n}\int_0^1(1-t)\,d^2h\big(b_n+t(a_n-b_n)\big)[d_n w_n,d_n w_n]\,dt.
$$
Since $b_n+t(a_n-b_n)\in K$ for all $t\in[0,1]$, \eqref{eq:bayart-D2} gives
$$
|R_n(h)|\le \frac12\,M\,d_n\,\|w_n\|^2.
$$
As $(w_n)$ is bounded and $d_n\to 0$, we get
$$
\sup_{\|h\|_L\le 1}|R_n(h)|\xrightarrow[n\to\infty]{}0.
$$
Moreover, by \eqref{eq:bayart-dh-lip} and since $b_n\to a$,
$$
\sup_{\|h\|_L\le 1}|(dh(b_n)-dh(a))(w_n)|\le M\|b_n-a\|\,\|w_n\|\xrightarrow[n\to\infty]{}0,
$$
and also
$$
\sup_{\|h\|_L\le 1}|dh(a)(w_n-w)|\le \|w_n-w\|\xrightarrow[n\to\infty]{}0.
$$
Combining these estimates shows that
$$
\sup_{\|h\|_L\le 1}\big|\langle m_n,h\rangle-dh(a)(w)\big|\xrightarrow[n\to\infty]{}0,
$$
that is, $\|m_n-e_{a,w}\|_{\G_0(B_X)}\to 0$.
\end{proof}

\subsection{Lipschitz image and compactness}
    
	Let us recall the next notation. 
    
\begin{definition}\label{lipimg}
		For $f \in \HL_0(B_X,Y)$, we define the set: 
	$$L_f:=\left\{\frac{f(x)-f(y)}{\|x-y\|} : x \neq y \in B_X\right\}.$$
\end{definition}

The set $L_f$ is sometimes called the Lipschitz image of $f$.

\begin{lemma} \label{lemma:Lf-to-fBX}
	Let $f \in \HL_0(B_X,Y)$. If $L_f$ is relatively (weakly) compact then $f(B_X)$ is relatively (weakly) compact.
\end{lemma}

\begin{proof}
	\noindent\textit{Norm compactness.}
	Taking $y=0$, we observe that
	$$\left\{\frac{f(x)}{\|x\|} : x \in B_X, x \neq 0 \right\} \subset L_f.$$
	From this we deduce that 
	$$\overline{f(B_X)} \subset \displaystyle\bigcup_{\lambda \in [0;1]} \lambda \overline{L_f}.$$ Indeed, pick $y \in \overline{f(B_X))}, y \neq 0$. There exists a sequence $(x_n) \subset B_X$ such that $f(x_n) \longrightarrow y$. Extracting if necessary, we may assume that $\|x_n\|\longrightarrow \lambda >0$ and we get $\lambda \frac{f(x_n)}{\|x_n\|} \longrightarrow y$. We conclude because $\bigcup\limits_{\lambda \in [0;1]} \lambda \overline{L_f}$ is clearly compact.
	
	\medskip
		\noindent\textit{Weak compactness.} 
		Assume that $L_f$ is relatively weakly compact in $Y$. Let $K:=\overline{L_f}^{w}$
		be the weak closure of $L_f$; then $K$ is weakly compact. Similarly as above, we have
		\begin{equation}\label{eq:image-in-union}
			f(B_X)\subset \bigcup_{\lambda\in[0,1]}\lambda K.
		\end{equation}
		We claim that $\bigcup\limits_{\lambda\in[0,1]}\lambda K$ is weakly compact.
		Consider the map
		\[
		\Phi\colon [0,1]\times K\longrightarrow Y,\qquad \Phi(\lambda,z):=\lambda z,
		\]
		where $[0,1]$ carries its usual topology and $K$ carries the weak topology of $Y$.
		For each $y^*\in Y^*$, the scalar function
		\[
		(\lambda,z)\longmapsto y^*(\Phi(\lambda,z))=y^*(\lambda z)=\lambda\,y^*(z)
		\]
		is continuous on $[0,1]\times K$ (product of continuous functions, since $z\mapsto y^*(z)$ is weakly continuous).
		Hence $\Phi$ is continuous as a map into $(Y,w)$.
		Since $[0,1]\times K$ is compact and $\Phi$ is continuous, its image
		\[
		\Phi([0,1]\times K)=\bigcup_{\lambda\in[0,1]}\lambda K
		\]
		is weakly compact. Together with \eqref{eq:image-in-union}, this shows that $f(B_X)$ is contained in a weakly compact set,
		hence $f(B_X)$ is relatively weakly compact.
\end{proof}

\section{When the range is an \texorpdfstring{$\ell_p$}{lp}-space}

We now turn to a more specific class of target spaces, which allows for a thiner analysis in infinite dimensions. More precisely, we will study holomorphic--Lipschitz maps $f\colon B_X\to \ell_p$, where $1\le p<\infty$ such that $f(B_X) \subset B_{\ell_p}$. Our starting point is Proposition~\ref{mol}, which reduces the problem to understanding the images of sequences of
normalized molecules, that is, sequences $(m_n)$ of the form
$$
m_n=\frac{\delta\big(f(x_n)\big)-\delta\big(f(y_n)\big)}{\|x_n-y_n\|},
\qquad x_n\neq y_n\in B_X.
$$

\subsection{Preparatory lemmas}

We will distinguish several regimes and use gliding-hump techniques. The next three lemmas follow this approach. The first one provides a convenient way to construct linear functionals
with pairwise disjoint supports.

\begin{lemma}\label{bg}
Let $1 \leq p<\infty$ and let $q$ be the conjugate exponent. Assume that $(x_n)\subset B_{\ell^{p}}$ satisfies
$x_n \to x$ weakly in $\ell_p$ (weakly$^*$ if $p=1$), and that $(x_n)$ is not relatively norm-compact.
Then there exist $\varepsilon>0$, a strictly increasing sequence $(n_k)$ in $\mathbb N$,
pairwise disjoint finite intervals $(I_k)$ of $\mathbb N$, and a sequence
$(x_k^*)\subset S_{\ell^{q}}$ such that $\supp(x_k^*)\subset I_k$ and
$$
\begin{cases}
\langle x_k^*,x_{n_k}\rangle>\varepsilon,\\[2pt]
\bigl|\langle x_k^*,x_{n_j}\rangle\bigr|<\dfrac{\varepsilon}{2^{k+2}}
\quad \text{for } j\neq k.
\end{cases}
$$
\end{lemma}

We only prove the case $1<p<\infty$, the proof being verbatim the same for $p=1$ and the weak$^*$ topology. 

\begin{proof}
Set $y_n:=x_n-x$. Then $y_n\overset{w}{\longrightarrow} 0$, and $(y_n)$ is not relatively norm-compact; in particular,
it is not norm-null. Hence, after passing to a subsequence, there exists $\varepsilon>0$ such that
\begin{equation}\label{eq:lower}
\|y_n\|_p\ge 4\varepsilon \qquad(n\in\mathbb N).
\end{equation}
For a finite $I\subset\mathbb N$, let $P_I:\ell_p\to\ell_p$ denote the coordinate projection onto $I$.
By the classical Bessaga--Pe{\l}czy\'nski gliding hump selection principle (see, e.g.,
\cite[Proposition~1.3.10]{AK}), passing to a further subsequence we may find pairwise disjoint
finite intervals $(I_k)$ and indices $(n_k)$ such that
\begin{equation}\label{eq:humps}
\|P_{I_k}y_{n_k}-y_{n_k}\|_p\le \varepsilon 2^{-k-4},\qquad
\|P_{I_k}y_{n_j}\|_p\le \varepsilon 2^{-k-4}\ \ (j\neq k),
\end{equation}
and (enlarging the starting point of $I_k$ if necessary) also
\begin{equation}\label{eq:x-tail}
\|P_{I_k}x\|_p\le \varepsilon 2^{-k-4}\qquad(k\in\mathbb N).
\end{equation}
Set $z_k:=P_{I_k}y_{n_k}$. From \eqref{eq:lower} and \eqref{eq:humps} we get
$\|z_k\|_p\ge 4\varepsilon-\varepsilon/16>3\varepsilon$.

For each $k$, choose $x_k^*\in S_{\ell_q}$ supported on $I_k$ such that
$\langle x_k^*,z_k\rangle=\|z_k\|_p$ (for instance, define $x_k^*=0$ on $\mathbb N\setminus I_k$ and on $I_k$ set
$x_k^*(m)=\operatorname{sgn}(z_k(m))\,|z_k(m)|^{p-1}/\|z_k\|_p^{p-1}$).
Then, using $x_{n_j}=y_{n_j}+x$ and \eqref{eq:humps}--\eqref{eq:x-tail}, we obtain
$$
\langle x_k^*,x_{n_k}\rangle
=\langle x_k^*,y_{n_k}\rangle+\langle x_k^*,x\rangle
\ge \langle x_k^*,z_k\rangle-\|y_{n_k}-z_k\|_p-\|P_{I_k}x\|_p
>3\varepsilon-2\cdot \varepsilon 2^{-k-4}>\varepsilon,
$$
and for $j\neq k$,
$$
\bigl|\langle x_k^*,x_{n_j}\rangle\bigr|
\le \|P_{I_k}y_{n_j}\|_p+\|P_{I_k}x\|_p
\le 2\cdot \varepsilon 2^{-k-4}
=\varepsilon 2^{-k-3}
<\frac{\varepsilon}{2^{k+2}}.
$$
\end{proof}

The next lemma provides convenient holomorphic--Lipschitz test functions on $B_{\ell_p}$, with controlled Lipschitz norm, adapted to disjoint supports.

\begin{lemma}\label{Fs}
Let $1 \leq p<\infty$ and let $q$ be the conjugate exponent. Let $s=(s_n)_{n \in \mathbb{N}} \in \ell_{\infty}$ and $(x_n^*) \subset S_{\ell_q}$ with disjoint supports. Let $m \in \mathbb{N}$ such that $(m-1)q \geq p$ ($m=1$ if $p=1$). Let:
$$F_s : x \in B_{\ell_p} \longmapsto \frac{1}{m}\sum_{n=1}^{+\infty} s_n(\langle x_n^*,x\rangle)^m \in \C.$$
Then, $F_s \in \HL_0(B_{\ell_p})$ with $\|F_s\|_L \leq \|s\|_{\infty}$. Moreover, if $\phi \in S_{\ell_q}$, the functional:
$$G_s : x \in B_{\ell_p} \longmapsto \frac{1}{m}\sum_{n=1}^{+\infty} s_n\langle \phi,x\rangle(\langle x_n^*,x\rangle)^m \in \C$$
is also in $\HL_0(B_{\ell_p})$ with $\|G_s\|_L \leq 2\|s\|_{\infty}$.
\end{lemma}

\begin{proof}
Because the $x_n^*$ have disjoint supports, $F_s$ is well defined. Then, $F_s$ is clearly holomorphic as an $m$-homogeneous polynomial, also, $F_s(0)=0$. It remains to show that $F_s$ is Lipschitz. To that end, we compute its differential, and use the fact that $\|F_s\|_L=\|dF_s\|_{\infty}$.

Let $x \in B_{\ell_p}$ and $h \in \ell_p$. 
$$dF_s(x)(h)=\sum_{n=1}^{+\infty} s_n(\langle x_n^*,x\rangle)^{m-1}\langle x_n^*,h\rangle.$$
The linear form $dF_s(x)$ identifies itself to an element of $\ell_q$. Let us compute its norm. Since the $x_n^*$ have disjoint support:
$$\|dF_s(x)\|_q^q=\sum_{n=1}^{+\infty} \|s_n(\langle x_n^*,x\rangle)^{m-1}x_n^*\|_q^q=\sum_{n=1}^{+\infty} |s_n|^q|\langle x_n^*,x\rangle|^{(m-1)q}.$$
Because $(m-1)q \geq p$ and $|\langle x_n^*,x\rangle|<1$:
$$\|dF_s(x)\|_q^q \leq \|s\|_{\infty}^q \sum_{n=1}^{+\infty} |\langle x_n^*,x\rangle|^p \leq\|s\|_{\infty}^q\sum_{n=1}^{+\infty} \|P_{\mathrm{supp}(x_n^*)}(x)\|_p^p \leq \|s\|_{\infty}\|x\|^p_p.$$
Hence, $\|dF_s\|_{\infty} \leq \|s\|_{\infty}$, so $F_s \in \HL_0(B_{\ell_p})$ and $\|F_s\|_L \leq \|s\|_{\infty}$.

For the second part, $G_s$ is clearly holomorphic and $G_s(0)=0$. Let us again compute the differential of $G_s$. For $x \in B_{\ell_p}$ and $h \in \ell_p$:
$$dG_s(x)(h)=\frac{1}{m}\sum_{n=1}^{+\infty} s_n\phi(h)\langle x_n^*,x\rangle^m+\sum_{n=1}^{+\infty}s_n\langle \phi,x\rangle\langle x_n^*,x\rangle^{m-1}x_n^*(h).$$
Since the $x_n^*$ have disjoint supports:
$$\|dG_s(x)\|_q^q \leq \frac{1}{m}\left\|\sum_{n=1}^{+\infty} s_n\langle x_n^*,x\rangle\phi\right\|_q^q+\sum_{n=1}^{+\infty} \|s_n\langle \phi,x\rangle\langle x_n^*,x\rangle^{m-1}x_n^*\|_q^q.$$
Let us get an upper bound for the second term as in the first part:
\begin{align*}\|dG_s(x)\|_q^q  &\leq \frac{1}{m}\left| \sum_{n=1}^{+\infty} s_n\langle x_n^*,x\rangle\right|^q+\|s\|_{\infty}^q\|x\|^p_p\\ &\leq \left(\sum_{n=1}^{+\infty} | s_n \langle x_n^*,P_{\mathrm{supp}(x_n^*)}(x)\rangle|\right)^q+\|s\|_{\infty}^q\|x\|_p^p\\
 &\leq \frac{1}{m}\|s\|_{\infty}^q\left(\sum_{n=1}^{+\infty} \|P_{\mathrm{supp}(x_n^*)}(x)\|_p\right)^q+\|s\|^q_{\infty}\|x\|_p^p\\
 &\leq \frac{1}{m}\|s\|_{\infty}^q\|x\|_p^q+\|s\|_{\infty}^q\|x\|_p^p \leq \left(1+\frac{1}{m}\right)\|s\|_{\infty}^q.
\end{align*}
Then, $G_s \in \HL_0(B_{\ell_p})$ and we obtain indeed $\|dG_s\|_{\infty}=\|G_s\|_L \leq 2\|s\|_{\infty}$.

\end{proof}

The next lemma provides a uniform separation estimate for powers on the unit disk: if two points are at least $\rho$ apart,
then one of finitely many powers separates them by a definite amount.

\begin{lemma}\label{anacomp}
Let $\rho\in ]0,2]$ and $M\geq 2$. There exists $C(\rho,M)>0$ such that, for every $\alpha,\beta\in\overline{\mathbb D}$ satisfying $|\alpha-\beta|\geq\rho$, there exists $m\in \{M,M+1\}$ such that $|\alpha^m-\beta^m|\geq C(\rho,M)$.
\end{lemma}

\begin{proof}
Let us assume that the assertion does not hold. Then, for all $n \in \mathbb{N}$, there exist $(\alpha_n),(\beta_n)$ two sequences in $\overline{\mathbb{D}}$ such that $|\alpha_n-\beta_n| \geq \rho$ and for all $m \in \{M_1,M_1+1\}$, we have $|\alpha_n^m-\beta_n^m| \leq \frac{1}{n}$.

Since $\overline{\mathbb{D}}$ is compact, up to passing to a subsequence, we may assume that $\alpha_n \underset{n \to \infty}{\longrightarrow} \alpha \in \overline{\mathbb{D}}$ and $\beta_n \underset{n \to \infty}{\longrightarrow} \beta \in \overline{\mathbb{D}}$. Then:
$$|\alpha-\beta| \geq \rho \text{ and } \forall m \in \{M_1,M_1+1\}, \alpha^m=\beta^m.$$
Therefore
$$\left(\frac{\alpha}{\beta}\right)^{M_1}=\left(\frac{\alpha}{\beta}\right)^{M_1+1}=1.$$
So $\frac{\alpha}{\beta}=1$ and $\alpha=\beta$ which contradicts the fact that $|\alpha-\beta|\geq \rho$.
\end{proof}

\subsection{Case analysis}\label{subsec:case-analysis}

Let $f\in \HL_0(B_X,\ell_p)$ be such that $f(B_X)\subset B_{\ell_p}$.
Our goal is to understand whether, for such maps, weak compactness of the canonical linearization $\widehat f\colon \G_0(B_X)\longrightarrow \G_0(B_{\ell_p})$ forces compactness. This property holds in the Lipschitz--free spaces setting (see \cite{ACp51}),
and we aim to identify precisely what remains true (and what fails) when the range is an $\ell_p$-space.

Our approach is by contradiction. Assume that $\widehat f$ is not compact. By Proposition~\ref{mol}, the set
$$
\widehat f(\mathcal M)
=\left\{\frac{\delta(f(x))-\delta(f(y))}{\|x-y\|}: x,y\in B_X,\ x\neq y\right\}
\subset \G_0(B_{\ell_p})
$$
is not relatively compact. Hence there exist sequences $(x_n)$ and $(y_n)$ in $B_X$ with $x_n\neq y_n$ such that
$$
m_n:=\frac{\delta(f(x_n))-\delta(f(y_n))}{\|x_n-y_n\|}
$$
has no norm-convergent subsequence in $\G_0(B_{\ell_p})$.
To rule out weak compactness, it will suffice to extract from $(m_n)$ a subsequence equivalent to the canonical basis of
$\ell_1$. Indeed, a sequence equivalent to the $\ell_1$-basis cannot be relatively weakly compact.

\medskip
\noindent\textit{Upper estimate.}
Since $\|m_n\|\le \Lip(f)$ for every $n$, the triangle inequality yields
\begin{equation}\label{eq:upper-l1}
	\forall\,c_1,\dots,c_N\in\C,\qquad
	\left\|\sum_{i=1}^N c_i m_i\right\|
	\le \Lip(f)\sum_{i=1}^N |c_i|.
\end{equation}
Thus, to prove $\ell_1$-equivalence, it only remains to obtain a uniform lower $\ell_1$-estimate.

\medskip
\noindent\textit{Lower estimate via test functions.}
We will construct a test function $F\in \HL_0(B_{\ell_p})$ such that, for all $N$ and all scalars $c_1,\dots,c_N$,
\begin{equation}\label{eq:lower-l1-plan}
	\left\|\sum_{i=1}^N c_i m_i\right\|
	\ge \frac{1}{\|F\|_L}\left|\left\langle F,\sum_{i=1}^N c_i m_i\right\rangle\right|
	\ge C\sum_{i=1}^N |c_i|,
\end{equation}
where $C>0$ is independent of $N$ and the coefficients. The functions produced in Lemma~\ref{Fs} are tailored to
establish \eqref{eq:lower-l1-plan}.

\medskip
\noindent\textit{Notation.}
Set
$$
a_n:=f(x_n),\qquad b_n:=f(y_n),\qquad d_n:=\|x_n-y_n\|.
$$
Then $a_n,b_n\in B_{\ell_p}$, $d_n\in(0,2]$, and
$$
m_n=\frac{\delta(a_n)-\delta(b_n)}{d_n}.
$$
Passing to a subsequence, we may assume that
$$
d_n\longrightarrow d\ge 0.
$$

\subsubsection{The macro regime: $d>0$}\label{subsubsec:macro}

Assume first that $d>0$. Passing to a subsequence, we may assume that
$$\forall n\in\N,\quad \frac{d}{2}\le d_n \le 2d.$$
Observe that $(a_n)$ and $(b_n)$ cannot both have norm-convergent subsequences:
indeed, if $a_n\to a$ and $b_n\to b$ in norm, then
$$
m_n=\frac{\delta(a_n)-\delta(b_n)}{d_n}\longrightarrow \frac{\delta(a)-\delta(b)}{d}
\qquad\text{in }\G_0(B_{\ell_p}),
$$
contradicting our choice of $(m_n)$. Since $\ell_p$ is reflexive, we may nevertheless assume (after extraction) that
both $(a_n)$ and $(b_n)$ are weakly convergent. Up to passing to a further subsequence, we are reduced to the following situations (and their symmetric variants):
\begin{enumerate}[(I)]
	\item $(a_n)$ is not relatively norm compact, while $b_n\to b$ in norm for some $b\in B_{\ell_p}$;
	\item both $(a_n)$ and $(b_n)$ fail to be relatively norm compact. We further split into
	\begin{enumerate}[(a)]
		\item $(a_n-b_n)$ is not relatively norm compact;
		\item $a_n-b_n\longrightarrow w$ in norm for some $w\in \ell_p$.
	\end{enumerate}
\end{enumerate}

Note that case (II)-(b) forces $w \neq 0$; otherwise $\|m_n\| = d_n^{-1}\|a_n-b_n\|_p \to 0$ which is impossible ($(m_n)$ does not have any convergent subsequence). 
We now treat these cases one by one.

\begin{lemma}\label{lem:macro-case-I-lower}
	Assume \textnormal{(I)}: $(a_n)$ is not relatively norm compact and $b_n\to b$ in norm for some $b\in B_{\ell_p}$. Let $m\in \N$ be such that $(m-1)q \geq p$. Then there exists $s\in \ell_\infty$ with $\|s\|_\infty\le 1$ such that, up to passing to a subsequence, for all $N \in \N$ and all scalars
	$c_1,\dots,c_N\in\C$ we have
	$$
	\left\langle F_s, \sum_{n=1}^{N} c_n m_n\right\rangle \ge \frac{\varepsilon^m}{8dm}\sum_{n=1}^{N} |c_n|,
	$$
	where $F_s$ is the test function constructed in Lemma~\ref{Fs} (and $\varepsilon>0$ is the separation parameter arising
		from the non-compactness of $(a_n)$).
\end{lemma}

\begin{proof}
	Since $(a_n)$ is weakly convergent but not relatively compact, we use Lemma~\ref{bg} to extract a subsequence of $(a_n)$, still denoted the same way for simplicity, and to find a sequence $(x_n^*) \subset S_{\ell_q}$ which verifies the consequences of the lemma for the sequence $(a_n)$. In Lemma \ref{bg}, we can construct the linear forms $(x_k^*)$ so that, for $k \neq n$, 
	$$
	|\langle x_k^*,a_n\rangle|^m \leq \frac{\varepsilon^m}{2^{k+4}}\quad \text{and} \quad  |\langle x_k^*,b_n\rangle|^m \leq \frac{\varepsilon^m}{2^{k+4}}.$$
	For $N \in \mathbb{N}, c_1,...,c_N \in \mathbb{C}$. Let, for all $n,k \in \mathbb{N}$,
	$$
	\Delta_{k,n}=\langle x_k^*,a_n\rangle^m-\langle x_k^*,b_n\rangle^m\quad \text{ and }\quad  s_k=\begin{cases}\frac{\overline{c_k}\overline{\Delta_{k,k}}}{|c_k||\Delta_{k,k}|} \text{ if } k \leq N\\
    0 \text{ if } k>N.
    \end{cases}
	$$ 
	We have that $s:=(s_n) \in \ell_{\infty}$, $\|s\|_{\infty}=1$ and for all $n \in \mathbb{N}$, $c_ns_n\Delta_{n,n}=|c_n||\Delta_{n,n}|$. Using the function $F_s$ defined in Lemma \ref{Fs}, and the triangle inequality, one has:
	\begin{align*}
		\left\langle F_s, \sum_{n=1}^{N} c_nm_n\right\rangle&=\sum_{n=1}^{N} c_n\langle F_s,m_n\rangle=\sum_{n=1}^{N}\frac{c_n}{md_n}\left(s_n\Delta_{n,n}+\sum_{k \neq n} s_k\Delta_{k,n}\right)\\    &=\sum_{n=1}^{N}\left(\frac{|c_n||\Delta_{n,n}|}{md_n}+\frac{c_n}{md_n}\sum_{k \neq n} s_k\Delta_{k,n}\right)\\ &\geq \sum_{n=1}^{N} |c_n|\left(\frac{\varepsilon^m}{2md_n}+\frac{1}{md_n}\sum_{k \neq n} -|\langle x_k^*,a_n\rangle|^m-|\langle x_k^*,b_n\rangle|^m\right).
	\end{align*}
	Since, for $k \neq n$, $|\langle x_k^*,a_n\rangle|^m \leq \frac{\varepsilon^m}{2^{k+4}}$ and $|\langle x_k^*,b_n\rangle|^m \leq \frac{\varepsilon^m}{2^{k+4}}$, we deduce
	$$\left\langle F_s, \sum_{n=1}^{N} c_nm_n\right\rangle\geq \sum_{n=1}^{N} |c_n|\left(\frac{1}{md_n}\frac{\varepsilon^m}{2} +\frac{1}{md_n}\left(-\frac{\varepsilon^m}{8}-\frac{\varepsilon^m}{8}\right)\right).$$
	Since $d_n \leq 2d$, we finally obtain
	$$\left\langle F_s, \sum_{n=1}^{N} c_nm_n\right\rangle \geq \frac{\varepsilon^m}{8dm}\sum_{n=1}^{N} |c_n|.$$
	So, in this case, we obtain the desired lower estimate with the constant $C=\frac{\varepsilon^m}{8m}$, recalling that $\|F_s\|_L \leq \|s\|_{\infty}=1$.
\end{proof}

\begin{lemma}\label{lem:macro-IIa}
	Assume \textnormal{(II)-(a)}: $(a_n)$ and $(b_n)$ are not $\|\cdot\|$-relatively compact and
	$(a_n-b_n)$ is not relatively $\|\cdot\|$-compact.
	Then, after passing to a subsequence, there exist $m\in\N$ with $(m-1)q\ge p$ and
	$s\in \ell_\infty$ with $\|s\|_\infty\le 1$ such that for all $N\in\N$ and all scalars
	$c_1,\dots,c_N\in\C$,
	$$
	\left\|\sum_{n=1}^{N} c_nm_n\right\|
	\ \ge\
	\frac{1}{\|F_s\|_L}\left|\left\langle F_s,\sum_{n=1}^{N} c_nm_n\right\rangle\right|
	\ \ge\
	\frac{1}{2dm}\left(C(\varepsilon)-\frac{\varepsilon^m}{2^{\ell}}\right)\sum_{n=1}^{N}|c_n|,
	$$
	where $\ell$ is chosen so that $C(\varepsilon)-\frac{\varepsilon^m}{2^\ell}>0$ and
	$C(\varepsilon)>0$ is the constant given by Lemma~\ref{anacomp}
	applied to $\rho=\varepsilon$.
\end{lemma}

\begin{proof}
	Since $(a_n-b_n)$ is not relatively norm compact, by Lemma~\ref{bg} we may find functionals
	$(x_n^*)\subset S_{\ell_q}$ with pairwise disjoint supports and a constant $\varepsilon>0$ such that
	$$
	\forall n\in \N, \quad|\langle x_n^*,a_n-b_n\rangle|\ge \varepsilon.
	$$
	Set $\alpha_n:=\langle x_n^*,a_n\rangle$ and $\beta_n:=\langle x_n^*,b_n\rangle$.
	Then $|\alpha_n|,|\beta_n|\le 1$ and
	$$
	|\alpha_n-\beta_n|=|\langle x_n^*,a_n-b_n\rangle|\ge \varepsilon.
	$$
	By Lemma~\ref{anacomp} (with $M$ being any integer so that $(M-1)q \geq p$), and after extracting and using a pigeonhole argument, there exist
	$m \geq \frac{p}{q}+1$ ($m=M$ or $m=M+1$) and $C(\varepsilon, \frac{p}{q}+1)>0$ such that
	$$\forall n\in \N, \quad
	|\alpha_n^{m}-\beta_n^{m}|\ge C(\varepsilon, \frac{p}{q}+1).
	$$
    For simplicity, we remove the dependence of $C$ with respect to $p$ and $q$ in the sequel.
	Since $(m-1)q\ge p$, Lemma~\ref{Fs} applies. Define
	$$
	\Delta_{k,n}:=\langle x_k^*,a_n\rangle^{m}-\langle x_k^*,b_n\rangle^{m}.
	$$
	Fix $N$ and scalars $c_1,\dots,c_N$. Choose $s\in\ell_\infty$ by
	$$
	s_n:=\frac{\overline{c_n}\,\overline{\Delta_{n,n}}}{|c_n|\,|\Delta_{n,n}|}\quad (1\le n\le N),\qquad s_n:=0\ (n>N),
	$$
	so that $\|s\|_\infty\le 1$ and $c_n s_n\Delta_{n,n}=|c_n|\,|\Delta_{n,n}|$ for $1\le n\le N$.
	Then, using $m_n=\frac{\delta(a_n)-\delta(b_n)}{d_n}$ and the definition of $F_s$,
	\begin{align*}
		\left\langle F_s,\sum_{n=1}^{N} c_nm_n\right\rangle
		&=\sum_{n=1}^{N}\frac{c_n}{md_n}\left(s_n\Delta_{n,n}+\sum_{k\neq n}s_k\Delta_{k,n}\right).
	\end{align*}
	Taking absolute values and using the triangle inequality,
	\begin{align*}
		\left|\left\langle F_s,\sum_{n=1}^{N} c_nm_n\right\rangle\right|
		&\ge \sum_{n=1}^{N}\frac{|c_n|}{md_n}\,|\Delta_{n,n}|
		-\sum_{n=1}^{N}\frac{|c_n|}{md_n}\sum_{k\neq n}|s_k|\,|\Delta_{k,n}|.
	\end{align*}
	The diagonal estimate gives $|\Delta_{n,n}|\ge C(\varepsilon)$.
	Moreover, by Lemma~\ref{bg} we may arrange (by prescribing a summable control sequence) that for $k\neq n$,
	$$
	|\langle x_k^*,a_n\rangle|^m\le \frac{\varepsilon^m}{2^{k+\ell+2}}
	\qquad\text{and}\qquad
	|\langle x_k^*,b_n\rangle|^m\le \frac{\varepsilon^m}{2^{k+\ell+2}},
	$$
	hence $|\Delta_{k,n}|\le |\langle x_k^*,a_n\rangle|^m+|\langle x_k^*,b_n\rangle|^m\le \frac{2\varepsilon^m}{2^{k+\ell+2}}$.
	Therefore,
	$$
	\sum_{k\neq n}|s_k|\,|\Delta_{k,n}|
	\le \sum_{k\neq n}|\Delta_{k,n}|
	\le \sum_{k\neq n}\frac{2\varepsilon^m}{2^{k+\ell+2}}
	\le \frac{\varepsilon^m}{2^\ell}.
	$$
	Consequently,
	$$
	\left|\left\langle F_s,\sum_{n=1}^{N} c_nm_n\right\rangle\right|
	\ge \sum_{n=1}^{N}\frac{|c_n|}{md_n}\left(C(\varepsilon)-\frac{\varepsilon^m}{2^\ell}\right).
	$$
	Finally, since $d_n\le 2d$ for all $n$, we obtain
	$$
	\left|\left\langle F_s,\sum_{n=1}^{N} c_nm_n\right\rangle\right|
	\ge \frac{1}{2dm}\left(C(\varepsilon)-\frac{\varepsilon^m}{2^\ell}\right)\sum_{n=1}^{N}|c_n|.
	$$
\end{proof}

\begin{lemma}\label{lem:macro-IIb}
	Assume \textnormal{(II)-(b)}: both $(a_n)$ and $(b_n)$ are not $\|\cdot\|$-relatively compact and
	$a_n-b_n\to w\neq 0$ in norm. Let $m\in\N$ be such that $(m-1)q\ge p$.
	Then, after passing to a subsequence, for all
	$N\in\N$ and all scalars $c_1,\dots,c_N\in\C$,
	$$
	\left\|\sum_{n=1}^{N} c_nm_n\right\|
	\ \ge\
	\frac{1}{\|G_s\|_L}\left|\left\langle G_s,\sum_{n=1}^{N} c_nm_n\right\rangle\right|
	\ \ge\
	\frac{\|w\|}{8dm}\Bigg(\frac{\varepsilon^m}{4}-\frac{m\varepsilon}{2^{\ell+2}}\Bigg)\sum_{n=1}^{N}|c_n|,
	$$
	where $\varepsilon>0$ is the separation parameter arising from the non-compactness of $(b_n)$ and $\ell$ is chosen so
	that $\frac{\varepsilon^m}{4}-\frac{m\varepsilon}{2^{\ell+2}}>0$.
\end{lemma}

\begin{proof}
	Since $a_n-b_n\to w\neq 0$, choose $\phi\in S_{\ell_q}$ with $\langle \phi,w\rangle=\|w\|$.
	Passing to a subsequence, we may assume
	$$
	\forall n \in \N, \quad \langle \phi,a_n-b_n\rangle\ge \frac12\|w\|.
	$$
	Apply Lemma~\ref{bg} to the non relatively compact sequence $(b_n)$ to obtain functionals
	$(x_k^*)\subset S_{\ell_q}$ with pairwise disjoint supports and $\varepsilon>0$ such that
	$$
	\forall n \in \N, \quad 
	|\langle x_n^*,b_n\rangle|^m\ge \varepsilon^m.
	$$
	and, moreover (after prescribing a summable control sequence),
	$$
	\forall k\neq n, \quad
	|\langle x_k^*,a_n-b_n\rangle|\le \frac{\varepsilon}{2^{k+\ell+2}}
	\quad\text{and}\quad
	|\langle x_k^*,a_n\rangle|^m\le \frac{\varepsilon^m}{2^{k+\ell+2}}
	$$
	Fix $N$ and $c_1,\dots,c_N$. Let $G_s$ be the test function from Lemma~\ref{Fs}:
	$$
	G_s(x)=\frac1m\sum_{k\ge1}s_k\,\langle \phi,x\rangle\,\langle x_k^*,x\rangle^m.
	$$
	For $n\le N$, define the diagonal coefficient
	$$
	D_n:=\langle \phi,a_n-b_n\rangle\,\langle x_n^*,b_n\rangle^m
	$$
	and choose
	$$
	s_n:=\frac{\overline{c_n}\,\overline{D_n}}{|c_n|\,|D_n|}\quad(1\le n\le N),\qquad s_k:=0\ (k>N),
	$$
	so that $\|s\|_\infty\le 1$ and $c_ns_nD_n=|c_n|\,|D_n|$. A direct computation yields
	$$
	\langle G_s,m_n\rangle=\frac{G_s(a_n)-G_s(b_n)}{d_n}
	=\frac{1}{md_n}\sum_{k\ge1}s_k\Big(\langle \phi,a_n\rangle\langle x_k^*,a_n\rangle^m
	-\langle \phi,b_n\rangle\langle x_k^*,b_n\rangle^m\Big).
	$$
	As in the original argument, we expand the $k=n$ term and use the mean value estimate
	$$
	\big|\langle x_n^*,a_n\rangle^m-\langle x_n^*,b_n\rangle^m\big|
	\le m\,|\langle x_n^*,a_n-b_n\rangle|.
	$$
	Using the choice of $s_n$ and the bounds coming from Lemma~\ref{bg}, one obtains
	$$
	\left|\left\langle G_s,\sum_{n=1}^{N} c_nm_n\right\rangle\right|
		\ge \sum_{n=1}^{N}\frac{|c_n|}{md_n}\left(\frac{\|w\|}{2}\varepsilon^m
		-\frac{\|w\|\varepsilon^m}{4}-\frac{m\varepsilon\|w\|}{2^{\ell+3}}-\frac{m\varepsilon\|w\|}{2^{\ell+3}}\right).
	$$
	Therefore,
	$$
	\left|\left\langle G_s,\sum_{n=1}^{N} c_nm_n\right\rangle\right|
		\ge \sum_{n=1}^{N}\frac{|c_n|}{2md_n}\,\|w\|\left(\frac{\varepsilon^m}{4}-\frac{m\varepsilon}{2^{\ell+2}}\right).
	$$
	Since $d_n\le 2d$ for all $n$, one has
	$$
	\left|\left\langle G_s,\sum_{n=1}^{N} c_nm_n\right\rangle\right|
	\ge \frac{\|w\|}{4dm}\left(\frac{\varepsilon^m}{4}-\frac{m\varepsilon}{2^{\ell+2}}\right)\sum_{n=1}^{N}|c_n|.
	$$
	Dividing by $\|G_s\|_L$ (with $\|G_s\|_L\le 2\|s\|_\infty\le 2$) yields the claim.
\end{proof}

\subsubsection{The micro regime: $d=0$}\label{subsubsec:micro}
Assume that $d=0$, that is, $d_n=\|x_n-y_n\|\to 0$. As in the macro regime, our aim is to obtain a uniform lower
$\ell_1$-estimate of the form \eqref{eq:lower-l1-plan}. Most of the arguments below apply to $\ell_p$-ranges for
$1\le p<\infty$; however, the last case will highlight the genuine specificity of the $\ell_1$-range.
Recall that $a_n,b_n\in B_{\ell_p}$ and
$$
m_n=\frac{\delta(a_n)-\delta(b_n)}{d_n},
\qquad
w_n:=\frac{a_n-b_n}{d_n}.
$$
Since $f$ is Lipschitz, $(w_n)$ is bounded. Up to passing to a subsequence, we reduce to the following situations:
\begin{enumerate}
	\item[(III)] $w_n\to w$ in norm, and one of the endpoint sequences (say $(b_n)$, or symmetrically $(a_n)$) is not relatively norm compact;
	\item[(IV)] $(w_n)$ is not relatively norm compact.
\end{enumerate}

	One may also notice that there is a seemingly missing case: $w_n\to w$ in norm and both $(a_n)$ and $(b_n)$ converge
	in norm in $B_{\ell_p}$. We claim that this situation cannot occur under our standing assumptions.
	Indeed, since $d_n\to 0$ and $a_n-b_n=d_n w_n$, we have $\|a_n-b_n\|_p\to 0$, and therefore the two sequences must
	converge to the same limit, say $a\in \overline{B}_{\ell_p}$.
	Consequently, the hypotheses of Lemma~\ref{Bayart} are satisfied, and we obtain
	$$
	m_n=\frac{\delta(a_n)-\delta(b_n)}{d_n}\longrightarrow e_{a,w}\qquad\text{in norm in }\G_0(B_{\ell_p}).
	$$
	In particular, $(m_n)$ has a norm-convergent subsequence, contradicting our assumption.
	Note also that (III) forces $w\neq 0$. Indeed, otherwise $(m_n)$ is norm convergent by Lemma~\ref{Bayart}).

\begin{lemma}\label{lem:micro-III-lower}
	Assume \textnormal{(III)}: $w_n\to w \neq 0$ in norm and $(b_n)$ is not relatively norm compact. 
	Then, after passing to a subsequence, for every $N\in\N$ and every choice of scalars $c_1,\dots,c_N\in\C$,
		there exists $s\in\ell_\infty$ with $\|s\|_\infty\le 1$ such that
	$$
	\left\|\sum_{n=1}^{N} c_n m_n\right\|
	\ \ge\
	\frac{1}{\|G_s\|_L}\left|\left\langle G_s,\sum_{n=1}^{N} c_n m_n\right\rangle\right|
	\ \ge\
	\left(\frac{\|w\|\varepsilon^m}{4m}-\frac{\varepsilon}{2^{\ell}}-\frac{\|w\|\varepsilon^m}{m\,2^\ell}\right)\sum_{n=1}^{N}|c_n|,
	$$
	where $m\in \N$ is such that $(m-1)q\geq p$, $\ell \in \N$ is chosen large enough so that the last bracket is positive,
	$G_s$ is the test function defined in Lemma~\ref{Fs}, and $\varepsilon>0$ is the separation parameter arising from the
	non-compactness of $(b_n)$.
\end{lemma}

\begin{proof}
	Choose $\phi \in S_{\ell_q}$ such that $\langle \phi,w\rangle=\|w\|$.
	Since $w_n\to w$ in norm, after discarding finitely many terms we may assume
	\begin{equation}\label{eq:phi-wn-lower-microIII}
		\forall n \in \N, \quad |\langle \phi,w_n\rangle|\ge \frac12\|w\|.
	\end{equation}
	Since $(b_n)$ is not relatively norm compact, by Lemma~\ref{bg}
	we can find, after passing to a subsequence, functionals $(x_k^*)\subset S_{\ell_q}$ with pairwise disjoint supports and
	a constant $\varepsilon>0$ such that
	\begin{equation}\label{eq:diag-b-microIII}
		\forall n \in \N, \quad |\langle x_n^*,b_n\rangle|\ge \varepsilon
	\end{equation}
	and, for a fixed integer $\ell\ge 1$ (to be chosen later),
	\begin{equation}\label{eq:offdiag-b-microIII}
		|\langle x_k^*,b_n\rangle|^m\le \frac{\varepsilon^m}{2^{k+\ell+2}},\quad\forall k\neq n.
	\end{equation}
	Moreover, since $w\in\ell_p$, by taking the supports far enough in the tail and using $w_n\to w$ in norm, we may also
	arrange that
	\begin{equation}\label{eq:w-small-microIII}
		|\langle x_k^*,w_n\rangle|\le \frac{\varepsilon}{2^{k+\ell+2}},\quad\forall k,n \in\N.
	\end{equation}
	
	\medskip
	Fix $N$ and coefficients $c_1,\dots,c_N$. Choose
	$$
	s_n:=\frac{\overline{c_n}\,\overline{D_n}}{|c_n|\,|D_n|}\quad(1\le n\le N),\qquad s_k:=0\ (k>N),
	$$
	so that $\|s\|_\infty\le 1$ and $c_ns_nD_n=|c_n|\,|D_n|$. Let $G_s$ be the test function from Lemma~\ref{Fs},
	$$
	G_s(x)=\frac{1}{m}\sum_{k\ge 1}s_k\,\langle \phi,x\rangle\,\langle x_k^*,x\rangle^m,
	\qquad x\in B_{\ell_p}.
	$$
	For $1\le n\le N$, set the diagonal scalar
	$$
	D_n:=\langle \phi,w_n\rangle\,\langle x_n^*,b_n\rangle^m,
	$$
	
	\medskip
	We now estimate the pairing. For each $n$,
	\begin{align*}
	\langle G_s,m_n\rangle&=\frac{G_s(a_n)-G_s(b_n)}{d_n}
	=\frac{1}{md_n}\sum_{k\ge 1}s_k\Big(\langle\phi,a_n\rangle\langle x_k^*,a_n\rangle^m
	-\langle\phi,b_n\rangle\langle x_k^*,b_n\rangle^m\Big)\\
	&=\frac{1}{m}\sum_{k\ge 1}s_k\,\langle \phi,w_n\rangle\,\langle x_k^*,b_n\rangle^m
	+\frac{1}{md_n}\sum_{k\ge 1}s_k\,\langle \phi,a_n\rangle\big(\langle x_k^*,a_n\rangle^m-\langle x_k^*,b_n\rangle^m\big).
	\end{align*}
	For the second sum we use the elementary estimate 
	$$
	\big|\langle x_k^*,a_n\rangle^m-\langle x_k^*,b_n\rangle^m\big|
	\le m\,|\langle x_k^*,a_n-b_n\rangle|
	=m\,d_n\,|\langle x_k^*,w_n\rangle|,
	$$
	hence
	$$
	\left|\frac{1}{md_n}\,\langle \phi,a_n\rangle\big(\langle x_k^*,a_n\rangle^m-\langle x_k^*,b_n\rangle^m\big)\right|
	\le |\langle x_k^*,w_n\rangle|.
	$$
	Therefore, using \eqref{eq:w-small-microIII},
	$$
		\left|\frac{1}{md_n}\sum_{k\ge 1}s_k\,\langle \phi,a_n\rangle\big(\langle x_k^*,a_n\rangle^m-\langle x_k^*,b_n\rangle^m\big)\right|
			\le \sum_{k\ge 1}|\langle x_k^*,w_n\rangle|
			\le \sum_{k\ge 1}\frac{\varepsilon}{2^{k+\ell+2}}
			\le \frac{\varepsilon}{2^\ell}.
	$$
	
	For the first sum, split the diagonal and off-diagonal parts. By \eqref{eq:phi-wn-lower-microIII} and
	\eqref{eq:diag-b-microIII}, and by the choice of $s_n$,
	$$
	\frac{1}{m}\,c_ns_n\,\langle \phi,w_n\rangle\,\langle x_n^*,b_n\rangle^m
	=\frac{|c_n|}{m}\,|D_n|
	\ge \frac{|c_n|}{m}\cdot \frac{\|w\|}{2}\,\varepsilon^m.
	$$
	Moreover, using \eqref{eq:offdiag-b-microIII} and $|\langle \phi,w_n\rangle|\le \|\phi\|\,\|w_n\|\le 2\|w\|$
	(after discarding finitely many terms),
	$$
	\left|\frac{1}{m}\sum_{k\neq n}s_k\,\langle \phi,w_n\rangle\,\langle x_k^*,b_n\rangle^m\right|
	\le \frac{2\|w\|}{m}\sum_{k\neq n}\frac{\varepsilon^m}{2^{k+\ell+2}}
		\le \frac{\|w\|\varepsilon^m}{m\,2^\ell}.
	$$
	
	Putting the estimates together and summing over $n=1,\dots,N$, we obtain
	$$
	\left|\left\langle G_s,\sum_{n=1}^{N} c_n m_n\right\rangle\right|
		\ge \left(\frac{\|w\|\varepsilon^m}{2m}-\frac{\|w\|\varepsilon^m}{m\,2^\ell}-\frac{\varepsilon}{2^\ell}\right)\sum_{n=1}^{N}|c_n|.
	$$
	Finally, since $\|G_s\|_L\le 2\|s\|_\infty\le 2$ (Lemma~\ref{Fs}), we get the claimed lower bound for the norm of
	$\sum\limits_{n=1}^N c_n m_n$. Choosing $\ell$ large enough makes the bracket positive.
\end{proof}

The next lemma is the only point in our argument where the case of an $\ell_p$-range with $1<p<\infty$ differs from the
	$\ell_1$ case. A first instance of this phenomenon already appeared in Example~\ref{ex:lp-compact-range-not-Lf}, and we
	will present a further one in Example~\ref{wcpNOTcp}.

\begin{lemma}\label{lem:l1-noncompact-slopes}
	Assume that $p=1$, that is, $f\in \HL_0(B_X,\ell_1)$ with $f(B_X)\subset B_{\ell_1}$. Assume \textnormal{(IV)}: $(w_n)$ is not relatively norm compact in $\ell_1$. Then $(m_n)$ admits a subsequence equivalent to the canonical basis of
	$\ell_1$. In particular, the sequence of molecules $(m_n)$ cannot be relatively weakly compact in $\G_0(B_{\ell_1})$.
\end{lemma}

\begin{proof}
	Since $\ell_1$ has the Schur property, a weakly Cauchy sequence in $\ell_1$ is norm Cauchy. Therefore, if $(w_n)$ is not
	relatively norm compact, it cannot have any weakly Cauchy subsequence.
	By Rosenthal's $\ell_1$ theorem (see, e.g., \cite[p.~252]{AK}), every bounded sequence in a Banach space either admits a
	weakly Cauchy subsequence or a subsequence equivalent to the canonical basis of $\ell_1$.
	Applying this to $(w_n)$, we conclude that $(w_n)$ admits a subsequence equivalent to the $\ell_1$ basis.
	
	Finally, since $\beta_{\ell_1}\colon \G_0(B_{\ell_1})\to \ell_1$ satisfies $\beta_{\ell_1}(m_n)=w_n$ and $\|\beta_{\ell_1}\|\le 1$, we have for every
	finite family of scalars $(c_i)$,
	$$
	\Big\|\sum_i c_i m_i\Big\|\ \ge\ \Big\|\beta\Big(\sum_i c_i m_i\Big)\Big\|
	=\Big\|\sum_i c_i w_i\Big\|.
	$$
	Hence any $\ell_1$ subsequence of $(w_n)$ yields an $\ell_1$ subsequence of $(m_n)$.
	In particular, $(m_n)$ cannot be relatively weakly compact in $\G_0(B_{\ell_1})$.
\end{proof}

\subsection{Assembling the argument}

From the preceding analysis we obtain our main result in the case of an $\ell_1$-range. For the definition and background information about (super) strict singularity and fixing a copy of $\ell_1$, we refer the reader to \cite{Lemay} and the references therein.

\begin{theorem}\label{thm:main}
	Let $f\in \HL_0(B_X,\ell_1)$ satisfy $f(B_X)\subset B_{\ell_1}$. The following assertions are equivalent:
	\begin{enumerate}[(i)]
		\item $\widehat f\colon \G_0(B_X)\to \G_0(B_{\ell_1})$ is compact;
		\item $\widehat f\colon \G_0(B_X)\to \G_0(B_{\ell_1})$ is weakly compact;
		\item $\widehat f\colon \G_0(B_X)\to \G_0(B_{\ell_1})$ is super strictly singular;
		\item $\widehat f\colon \G_0(B_X)\to \G_0(B_{\ell_1})$ is strictly singular;
		\item $\widehat f$ does not fix any copy of $\ell_1$.
	\end{enumerate}
\end{theorem}

\begin{proof}
	The implications
	$$
	\textnormal{(i)}\implies\textnormal{(iii)}\implies\textnormal{(iv)}\implies\textnormal{(v)}
	\qquad\text{and}\qquad
	\textnormal{(i)}\implies\textnormal{(ii)}
	$$
	are standard. It remains to show that \textnormal{(ii)} implies \textnormal{(i)}, and that \textnormal{(v)} implies \textnormal{(ii)}.
	Assume that $\widehat f$ is not weakly compact. Then, by Proposition~\ref{mol}, the set $\widehat f(\mathcal M)$ is not relatively weakly compact. Our case analysis (macro regime $d>0$, micro regime
	$d=0$) shows that in this situation one can extract a subsequence of $\widehat f(\mathcal M)$ equivalent to the canonical
	basis of $\ell_1$. In particular, $\widehat f$ fixes a copy of $\ell_1$, so \textnormal{(v)} fails.
	Equivalently, \textnormal{(v)} implies \textnormal{(ii)}.
	
	Finally, the same case analysis shows that whenever $\widehat f$ fails to be compact, one can extract an $\ell_1$-sequence
	in $\widehat f(\mathcal M)$, and hence $\widehat f$ fails to be weakly compact. Thus \textnormal{(ii)} implies
	\textnormal{(i)}.
\end{proof}

\section{General results}

\subsection{Preparatory results}

\begin{thm}\label{CScomp}
Let $f \in \HL_0(B_X,Y)$ such that $f(B_X)\subset B_Y$. If 
\begin{enumerate}[$(i)$]
	\item $\overline{f(B_X)} \subset B_Y$, and 
	\item $\displaystyle L_f$ is relatively compact,
\end{enumerate} 
then $\widehat{f} : \G_0(B_X) \longrightarrow \G_0(B_Y)$ is compact.
\end{thm}

\begin{proof}
First, thanks to Lemma~\ref{lemma:Lf-to-fBX}, $f(B_X)$ is relatively compact. By Proposition~\ref{mol}, one just need to show that 
$$\left\{\frac{\delta(f(x))-\delta(f(y))}{\|x-y\|} : x,y \in B_X, x \neq y\right\}$$
is relatively compact in $\G_0(B_Y)$. Pick $(x_n)$ and $(y_n)$ two sequences of $B_X$ with $x_n \neq y_n$. Denote, for all $n \in \mathbb{N}$, $$m_n=\frac{\delta(f(x_n))-\delta(f(y_n))}{\|x_n-y_n\|}\quad \text{ and }\quad  w_n=\frac{f(x_n)-f(y_n)}{\|x_n-y_n\|}.$$ 
Since $f(B_X)$ is relatively compact, extracting if necessary, we may assume that $(f(x_n))$ converges to $a \in B_Y$ and that $(f(y_n))$ converges to $b \in B_Y$. By a further extraction, we may assume that $(\|x_n-y_n\|)$ goes to $\rho \geq 0$. If $\rho>0$, then, $(m_n)$ converges to $\frac{\delta(a)-\delta(b)}{\rho}$. Hence, we may assume that $(\|x_n-y_n\|)$ goes to zero. This forces $a=b$ because $f$ is Lipschitz. With a last extraction, we may assume that there exists $w \in Y$ such that
$$\lim_{n \rightarrow +\infty} \frac{f(x_n)-f(y_n)}{\|x_n-y_n\|}=w.$$
We can now use Lemma \ref{Bayart} to conclude that 
$$\frac{\delta(f(x_n))-\delta(f(y_n))}{\|x_n-y_n\|} \underset{n \to \infty}{\longrightarrow} e_{a,w},$$
where $e_{a,w}(f)=df(a)(w)$, and that $\widehat{f}$ is compact.
\end{proof}

\begin{remark}
\label{rmkDf}	
Let
$$
L_f:=\left\{\frac{f(x)-f(y)}{\|x-y\|}:x,y\in B_X,\ x\neq y\right\}
\qquad\text{and}\qquad
D_f:=\{df(x)(u):x\in B_X,\ u\in S_X\}.
$$
Then $D_f\subset \overline{L_f}$, because for each $x\in B_X$ and $u\in S_X$ we have
$$
df(x)(u)=\lim_{t\to 0}\frac{f(x+tu)-f(x)}{t}\in \overline{L_f}.
$$
Conversely, for $x\neq y$ in $B_X$ we have the segment formula
$$
\frac{f(x)-f(y)}{\|x-y\|}
=\int_0^1 df\big(y+t(x-y)\big)\!\left(\frac{x-y}{\|x-y\|}\right)\,dt,
$$
so $L_f\subset \overline{\conv}(D_f)$. In particular, $\overline{L_f}\subset \overline{\conv}(D_f)$.
Therefore $L_f$ is relatively (weakly) compact if and only if $D_f$ is relatively (weakly) compact.
\end{remark}

In order to give some necessary conditions for compactness, let us introduce the barycenter map:

\begin{lemma}\label{ineq}
	The barycenter map $\beta_X\colon \G_0(B_X)\to X$, given by $\beta_X(\delta(u))=u$ for every $u\in B_X$, is a bounded operator with $\|\beta\|\le 1$.
	In particular, for $N \geq 1$, $x_1,\ldots ,x_N \in B_X, \; a_1,\ldots,a_N \in \mathbb{C}$, one has
	$$\left\|\sum_{k=1}^{N} a_k\delta(x_k)\right\|_{\G_0(B_X)} \geq \left\|\sum_{k=1}^{N} a_kx_k\right\|_X.$$
\end{lemma}

\begin{proof}
	Simply note that $\beta_X$ is the canonical extension of $Id_{B_X}: x \in B_X \mapsto x \in B_X $ using the universal extension property of $\G_0(B_X)$, that is $\beta_X = \overline{Id_{B_X}}$.
\end{proof}

We shall use in the sequel that $\beta$, as a bounded linear map, is also weak to weak continuous.

\begin{prop}\label{CNcomp}
	Let $f\in \HL_0(B_X,Y)$.
	If $\widehat f$ is (weakly) compact, then $L_f$ is relatively (weakly) compact
	in $Y$. In particular, $\overline{f(B_X)}$ is (weakly) compact. 
\end{prop}

\begin{proof}
Assume that $\widehat f$ is (weakly) compact. Since $\Mol$ is bounded, its image
			$\widehat f(\Mol)$ is relatively (weakly) compact in $\G_0(B_Y)$.
			The barycenter map $\beta_Y$ is norm (weak-to-weak) continuous; therefore
			$\beta_Y(\widehat f(\Mol))$ is relatively (weakly) compact in $Y$.
			Using $L_f=\beta_Y\big(\widehat f(\Mol(B_X))\big)$, we conclude that $L_f$ is relatively (weakly) compact. The last assertion follows from Lemma~\ref{lemma:Lf-to-fBX}.
\end{proof}

If we suppose that $\hat f$ is a compact operator, we obtain further informations on the image of $f(B_X)$.

\begin{thm}\label{CNcomp2}
	Let $f \in \HL_0(B_X,Y)$ such that $f(B_X) \subset B_Y$. Assume that $\widehat{f} : \G_0(B_X) \to \G_0(B_Y)$ is compact. Then
	for all $y^* \in S_{Y^*},$ for all sequences $(z_j)\subset B_X$ such that  $y^* \circ f(z_j) \to 1$, then $ \|y^* \circ df(z_j)\|\to 0$.
	In particular,  $\overline{f(B_X)} \subset B_Y$.
\end{thm}

\begin{proof}
		Since $\widehat f$ is compact, Schauder's theorem implies that its adjoint $C_f$ is compact on $\HL_0(B_Y)$.

		Assume that there exist $y^*\in S_{Y^*}$, a sequence $(z_j)\subset B_X$ and $\varepsilon>0$ such that
		$$
		y^*(f(z_j))\xrightarrow[j\to\infty]{}1
		\qquad\text{and}\qquad
		\|y^*\circ df(z_j)\|\ge \varepsilon\quad\text{for all }j.
		$$
		Choose integers $n_j\to\infty$ so that $|y^*(f(z_j))|^{n_j}\ge \tfrac12$ for all large $j$.
		Define
		$$
		h_j(z):=\frac{\langle y^*,z\rangle^{\,n_j+1}}{n_j+1},\qquad z\in B_Y.
		$$
		Then $h_j\in \HL_0(B_Y)$ and
		$$
		dh_j(z)(\cdot)=\langle y^*,z\rangle^{\,n_j}\langle y^*,\cdot\rangle,
		\qquad\text{hence}\qquad
		\|h_j\|_L=\sup_{z\in B_Y}\|dh_j(z)\|\le 1.
		$$
		Therefore $(h_j)$ is bounded in $\HL_0(B_Y)$, and since $C_f$ is compact we may assume (after extracting) that
		$h_j\circ f$ converges in $\HL_0(B_X)$. For each fixed $x\in B_X$, we have $|y^*(f(x))|<1$, hence
		$h_j(f(x))\to 0$, and therefore $h_j\circ f\to 0$ in $\HL_0(B_X)$.
		
		On the other hand,
		$$
		\|h_j\circ f\|_L\ge \|d(h_j\circ f)(z_j)\|.
		$$
		By the chain rule,
		$$
		d(h_j\circ f)(z_j)=dh_j(f(z_j))\circ df(z_j)=\langle y^*,f(z_j)\rangle^{\,n_j}\, (y^*\circ df(z_j)),
		$$
		and thus
		$$
		\|h_j\circ f\|_L\ge |\langle y^*,f(z_j)\rangle|^{n_j}\,\|y^*\circ df(z_j)\|
		\ge \frac{\varepsilon}{2}
		$$
		for all large $j$. This contradicts $h_j\circ f\to 0$ in $\HL_0(B_X)$. Hence the first assertion holds.
		
The second assertion follows from Proposition \ref{prop-images-relation}.
\end{proof}

\subsection{When \texorpdfstring{$\dim(Y)<\infty$}{dim(Y)<infty} and consequences}

When the range space of $f$ is finite dimensional, one has the equivalence between compactness and weak compactness of $\widehat{f}$. 

\begin{cor}\label{cor:Finite-dim-range}
	Let $f\in \HL_0(B_X,Y)$ satisfy $f(B_X)\subset B_{Y}$ and $\dim(Y)<\infty$. Then the following assertions are equivalent:
	\begin{enumerate}[(i)]
		\item $\widehat f\colon \G_0(B_X)\to \G_0(B_{Y})$ is compact;
		\item $\widehat f\colon \G_0(B_X)\to \G_0(B_{Y})$ is weakly compact;
		\item $\widehat f\colon \G_0(B_X)\to \G_0(B_{Y})$ is super strictly singular;
		\item $\widehat f\colon \G_0(B_X)\to \G_0(B_{Y})$ is strictly singular;
		\item $\widehat f$ does not fix any copy of $\ell_1$.
	\end{enumerate}
\end{cor}

\begin{proof}
We only have to prove $(v) \implies (i)$. Assume that $\widehat f$ is not compact. Since $Y$ is finite dimensional, $L_f$ is relatively compact. Hence, by Theorem~\ref{CScomp}, we necessarily have $\overline{f(B_X)}\cap S_Y\neq\varnothing$. Let $y\in\overline{f(B_X)}\cap S_Y$ and choose $y^*\in S_{Y^*}$ such that $y^*(y)=1$. Set $g:=y^*\circ f\in \HL_0(B_X,\mathbb C)$. Then $1\in\overline{g(B_X)}$. We claim that $\widehat g$ is not compact. Indeed, otherwise Theorem~\ref{CNcomp2} would imply property {\rm (P3)} for $g$, contradicting Lemma~\ref{contp4}.

Now consider the canonical isometric embedding $J:\lambda \in \C \mapsto \lambda e_1 \in \ell_1$ and set $h:=J\circ g$. Then $h\in \HL(B_X , B_{\ell_1})$. Let $P:\ell_1\to\mathbb C$ be the first-coordinate projection. Since $P\circ J=Id_{\mathbb C}$, from $P \circ h = P \circ J \circ g = g$ we deduce $\widehat g=\widehat P\circ\widehat h.$ Therefore $\widehat h$ is not compact. By Theorem~\ref{thm:main}, $\widehat h$ fixes a copy of $\ell_1$. Since $\widehat h = \widehat J\circ\widehat{y^*}\circ\widehat f$, it follows that $\widehat f$ itself fixes a copy of $\ell_1$.
\end{proof}

We derive the following useful result.

\begin{proposition}\label{lem:wkcomp-strict-contraction}
		Let $f\in \HL_0(B_X,Y)$ satisfy $f(B_X)\subset B_Y$ and assume that $\widehat f$ is weakly compact.
		Then $\overline{f(B_X)} \subset B_Y$. In particular, we obtain the same conclusion if $\widehat{f}$ is compact.
\end{proposition}

\begin{proof}
Assume that $\widehat{f}$ is weakly compact and $\overline{f(B_X)} \cap S_Y \neq \emptyset$. Let $y \in \overline{f(B_X)} \cap S_Y$ and let $y^* \in Y^*$ such that $\langle y^*,y\rangle=1$. Consider $g=y^* \circ f$. Then $g \in HL_0(B_X, \mathbb{C})$ and $g(B_X) \subset \mathbb{D}$. We also have that $\widehat{g}=\widehat{y^*} \circ \widehat{f}$, and since $\widehat{f}$ is weakly compact, so is $\widehat{g}$. Since $\mathbb{C}$ is finite dimensional, using Theorem \ref{thmC} (see Corollary \ref{cor:Finite-dim-range}), $\widehat{g}$ is actually compact. By Theorem \ref{CNcomp2}, and using the fact that the complex dual of $\mathbb{C}$ is exactly the set of all the linear maps of the form $z \longmapsto az$, $a \in \mathbb{C}$, we have that:
$$\forall a \in \mathbb{C}, \forall (z_j) \subset B_X, ag(z_j) \underset{j \rightarrow +\infty}{\longrightarrow} 1 \implies \|adg(z_j)\| \underset{j \rightarrow +\infty}{\longrightarrow} 0.$$
Taking $a=1$, we get:
$$\forall (z_j) \subset B_X, y^* \circ f(z_j) \underset{j \rightarrow +\infty}{\longrightarrow} 1 \implies \|y^* \circ df(z_j)\| \underset{j \rightarrow +\infty}{\longrightarrow} 0.$$
But now, apply Lemma \ref{contp4} to get a contradiction.
\end{proof}

The following result has already been obtained (independently) in \cite[Corollary 4.15]{LuisCarlos}. It is straightforward thanks to Theorem \ref{CNcomp2} and Theorem \ref{CScomp}.

\begin{cor}\label{dimfinie}
Let $f \in \HL_0(B_X,Y)$ satisfy $f(B_X)\subset B_Y$. Assume that $\dim(X)<\infty$ and $\dim(Y)<\infty$. Then $\widehat{f}$ is compact if and only if $\|f\|_{\infty}<1$.
\end{cor}

\subsection{Metric characterization of compactness}

One can now prove the metric characterization of compactness that constitutes Theorem \ref{thmD}. We recall its statement:

\begin{theorem}\label{caractcomp}
Let $X, Y$ be complex Banach spaces and $f \in HL_0(B_X,Y)$ such that $f(B_X) \subset B_Y$. Then: 
$$\widehat{f} : \G_0(B_X) \longrightarrow \G_0(B_Y) \text{ is compact } \iff \overline{L_f} \text{ is compact and } \|f\|_{\infty}<1.$$
\end{theorem}

\begin{proof}
Assume that $\widehat{f}$ is compact, then $\widehat{f}$ is weakly compact and by Proposition \ref{lem:wkcomp-strict-contraction}, we have that $\overline{f(B_X)} \subset B_Y$. Moreover, Proposition \ref{CNcomp} ensures that $\overline{L_f}$ is compact, hence $\overline{f(B_X)}$ is compact (Lemma \ref{lemma:Lf-to-fBX}) , and so $\|f\|_{\infty}<1$ (see Proposition \ref{prop-images-relation}).

Conversely, if $\|f\|_{\infty}<1$, then $\overline{f(B_X)} \subset B_Y$ and just apply Theorem \ref{CScomp}.
    
\end{proof}

\subsection{Metric characterization of weak compactness}

Let us start with the following sufficient condition for weak compactness:

\begin{thm}\label{CSwcomp}
	Let $X,Y$ Banach spaces with $Y$ reflexive. Let $f \in \HL_0(B_X,Y)$ such that $f(B_X) \subset B_Y$. Assume that:
	\begin{enumerate}
		\item[(i)] $\overline{f(B_X)} \subset B_Y$, and
		\item[(ii)] $f(B_X)$ is relatively compact.
	\end{enumerate}
	Then, $\widehat{f}$ is weakly compact.
\end{thm}

\begin{proof}
	In the light of Proposition \ref{mol}, let us show that $$m_n=\frac{\delta(f(x_n))-\delta(f(y_n))}{\|x_n-y_n\|},$$ 
	with $x_n \neq y_n \in B_X$ for all $n \in \mathbb{N}$, admits a weakly convergent subsequence. For convenience, let us denote $a_n=f(x_n)$, $b_n=f(y_n)$ and $w_n=\frac{a_n-b_n}{\|x_n-y_n\|}$. 
	
	Because of $(i)$ and $(ii)$, we may assume up to passing to a subsequence, that $a_n \longrightarrow a \in B_Y$ and $b_n \longrightarrow b \in B_Y$ for the norm topology. In addition, because $(x_n)$ and $(y_n)$ are sequences in $B_X$, we also have, up to an extraction, $\|x_n-y_n\| \longrightarrow d \geq 0$. Moreover, since $Y$ is reflexive, up to passing to a subsequence, $w_n \to w$ for the weak topology. Let us consider two cases:
	\begin{enumerate}
		\item[(1)] If $d>0$, then it is clear that $(m_n)$ converges for the norm topology (and so for the weak topology) to $\frac{\delta(a)-\delta(b)}{d}$.
		\item[(2)] If $d=0$, then $a=b$ because $f$ is Lipschitz. Then, we can use Lemma~\ref{Bayart} to show that $m_n \to e_{a,w}$ weakly.
	\end{enumerate}
\end{proof}

Of course, Theorem~\ref{CSwcomp} should be compared with Theorem~\ref{CScomp}. Besides the reflexivity assumption on
$Y$, the only difference in the hypotheses is that the condition ``$L_f$ is relatively compact'' in
Theorem~\ref{CScomp} is replaced by the weaker requirement that ``$f(B_X)$ is relatively compact''.
Note that when $Y$ is reflexive, every bounded subset of $Y$ is relatively weakly compact; in particular, both $L_f$ and $f(B_X)$ are automatically relatively weakly compact. Nevertheless, as the example below shows, relative compactness of $f(B_X)$ does not imply relative compactness of $L_f$. 

\begin{example}\label{ex:lp-compact-range-not-Lf}
	Let $1\le p<\infty$ and set $X=Y=\ell_p$. Define, for $x=(x_n)\in B_{\ell_p}$,
	\[
	f(x):=\Big(\frac{x_n^{\,n^2}}{n^2}\Big)_{n\ge2}.
	\]
	Then $f$ is clearly 1-Lipschitz and holomorphic, so $f\in \HL_0(B_{\ell_p},\ell_p)$. Moreover,
	\[
	f(B_{\ell_p})\subset K:=\left\{u\in \ell_p:\ \forall n, \ |u_n|\le \frac{1}{n^2}\right\}.
	\]
	Since $\left(\frac{1}{n^2}\right)\in \ell_p$, the tails of $K$ are uniformly small:
	\[
	\sup_{u\in K}\|u-P_Nu\|_p\le \Big(\sum_{n>N}n^{-2p}\Big)^{1/p}\xrightarrow[N\to\infty]{}0,
	\]
	where $P_N$ is the projection onto the $N$ first coordinates. Since $P_N(K)$ is compact in $\C^N$, $K$ is compact. Thus $f(B_{\ell_p})$ is relatively compact. 
	
	Now, set $t_n:=\frac{1}{n^2}$, $r_n:=1-\frac{1}{n^4}$, $x^{(n)}:=r_ne_n$, $y^{(n)}:=(r_n-t_n)e_n$, where $(e_n)$ is the standard unit vector basis of $\ell_p$.
	Then $x^{(n)},y^{(n)}\in B_{\ell_p}$ and 
	$$\|x^{(n)}-y^{(n)}\|_p=t_n\to 0.$$
	The slopes $w_n:=\dfrac{f(x^{(n)})-f(y^{(n)})}{\|x^{(n)}-y^{(n)}\|_p}\in L_f$ satisfy
	\[
	w_n=\frac1{n^2}\cdot\frac{r_n^{n^2}-(r_n-t_n)^{n^2}}{t_n}\,e_n
	=\lambda_n^{n^2-1}e_n
	\]
	for some $\lambda_n\in(r_n-t_n,r_n)$ (mean value theorem). Since $\lambda_n \geq 1-\frac{1}{n^2}-\frac{1}{n^4}$ so that $\lambda_n^{n^2-1} \geq (1-\frac{1}{n^2}-\frac{1}{n^4})^{n^2-1}\to e^{-1}$.
	Hence $\liminf \|w_n\|_p\ge \frac{e^{-1}}{2}$ for all large $n$. In particular $\|w_n-w_m\|_p\ge 2^{\frac{1}{p}} \frac{e^{-1}}{2}>0$ for $n\neq m$ large enough, and $L_f$ is not relatively compact.
\end{example}

Note that, for $1<p<\infty$, this already yields an example of a weakly compact operator $\widehat f$ (by Theorem~\ref{CSwcomp})
which fails to be compact (see Proposition~\ref{CNcomp}). Another example will be given later.

One can now prove the metric characterization of weak compactness in the case where the target space is $\ell_p$ that constitutes Theorem \ref{thmE}.

\begin{theorem}\label{recip}
	Let $1 <p<\infty$ and let $f \in \HL_0(B_X,\ell_p)$ such that $f(B_X) \subset B_{\ell_p}$ and let $\widehat{f} : \G_0(B_X) \longrightarrow \G_0(B_{\ell_p})$ the associated holomorphic-Lipschitz operator. Then:
    $$\widehat{f} \text{ is weakly compact } \iff \overline{f(B_X)} \text{ is compact and } \|f\|_{\infty}<1.$$

    Moreover, if $p=1$, one implication remains true:
    $$\widehat{f} \text{ is weakly compact } \implies \overline{f(B_X)} \text{ is compact and } \|f\|_{\infty}<1,$$
    but the converse is not true.
\end{theorem}

\begin{proof}
	Assume that $\widehat{f}$ is weakly compact and $\overline{f(B_X)}$ is not compact. Then there exist $\varepsilon>0$ and a sequence $(x_n)\subset B_X$ such that
	$(f(x_n))$ is $\varepsilon$-separated. In particular, $(x_n)$ cannot converge to $0$ (otherwise $f(x_n)\to 0$).
	Passing to a subsequence, we may therefore assume that $\|x_n\|\to d$ for some $d\in(0,1]$. Set
	$$
	m_n:=\frac{\delta(f(x_n))-\delta(0)}{\|x_n\|}\in \G_0(B_{\ell_p})
	\qquad(n\in\N).
	$$
	Since $d>0$, this is a macro situation. Arguing as in the macro regime of Section~\ref{subsubsec:macro} (with $b_n=0$
	and $a_n=f(x_n)$), one extracts a subsequence of $(m_n)$ equivalent to the canonical basis of $\ell_1$.
	In particular, $(m_n)$ is not relatively weakly compact in $\G_0(B_{\ell_p})$, which contradicts the weak compactness of
	$\widehat f$. Hence $f$ must be compact.

	To finish the proof of the first part, we apply Proposition~\ref{lem:wkcomp-strict-contraction} to obtain that $\overline{f(B_X)} \subset B_{\ell_p}$ and since $\overline{f(B_X)}$ is compact, Proposition \ref{prop-images-relation} ensures that $\|f\|_{\infty}<1$. Hence, we have the implication.

    If $p \neq 1$, then $\ell_p$ is reflexive. If $\|f\|_{\infty}<1$, then $\overline{f(B_X)} \subset B_{\ell_p}$ and Theorem~\ref{CSwcomp} gives the implication. The converse is not true because we know that when the range space is $\ell_1$, compactness and weak compactness are equivalent. So if the converse was true, it would mean that $\overline{f(B_X)}$ compact and $\overline{L_f}$ compact would be equivalent (see Theorem \ref{caractcomp}) but it is not the case, see Example \ref{ex:lp-compact-range-not-Lf}.
\end{proof}

\begin{example}\label{Tsirelson}
This characterization cannot be generalized to every reflexive spaces. Indeed, consider $X=T^*$ the original Tsirelson space (see \cite[Theorem 10.3.2]{AK}. It is known that $X$ is reflexive. In \cite[Example 4.20]{LuisCarlos}, the authors considered $f : x \in X \longmapsto \frac{x}{2} \in X$. One has $f \in HL_0(B_X,X)$ with $f(B_X) \subset B_X$. The authors of \cite{LuisCarlos} showed that $\widehat{f}$ is weakly compact. However, $f(B_X)=\frac{1}{2}B_X$ is not relatively compact, since $X$ is infinite dimensional.
\end{example}

Let us now provide another weakly compact but not compact $\widehat{f}$, in the case $1<p<\infty$.
\begin{example}\label{wcpNOTcp}
Fix $1<p<\infty$ and set $\lambda:=\left(\frac{1}{n}\right)_{n\ge 2}\in \ell_p$.
Choose a constant $\alpha>0$ such that $\|\alpha\lambda\|_p<1$ (for instance,
$\alpha:=\big(2\|\lambda\|_p\big)^{-1}$).
Define the diagonal map
$$
f\colon B_{\ell_p}\to \ell_p,\qquad
f(x):=\Big(\frac{\alpha}{n}\,x_n^{\,n}\Big)_{n\ge 2}.
$$
Then $f$ is holomorphic, $f(0)=0$, and for $x,h\in B_{\ell_p}$ we have
$$
df(x)(h)=\big(\alpha\, x_n^{\,n-1}h_n\big)_{n\ge 2}.
$$
Since $|x_n|\le \|x\|_p\le 1$ for each $n$, we get $\|df(x)\|_{\infty}\le \alpha$ for all $x\in B_{\ell_p}$, hence
$f\in \HL_0(B_{\ell_p})$.

\smallskip
\noindent\emph{$\widehat f$ is not compact.}
Let $x_n:=\big(1-\frac1n\big)e_n$ and $y_n:=e_n$ for $n\ge 2$. Then $x_n,y_n\in B_{\ell_p}$ and
$d_n:=\|x_n-y_n\|_p=\frac1n\to 0$.
Moreover,
$$
f(x_n)=\frac{\alpha}{n}\Big(1-\frac1n\Big)^n e_n,
\qquad
f(y_n)=\frac{\alpha}{n}e_n.
$$
Set
$$
v_n:=\frac{\delta(f(x_n))-\delta(f(y_n))}{d_n}\in \G_0(B_{\ell_p}).
$$
Testing against the coordinate functional $e_n^*\in (\ell_p)^*\equiv \ell_q\subset \HL_0(B_{\ell_p})$, we obtain
$$
\langle v_n,e_n^*\rangle=\frac{e_n^*(f(x_n))-e_n^*(f(y_n))}{d_n}
=\alpha\Big(\Big(1-\frac1n\Big)^n-1\Big).
$$
Hence
$$
\|v_n-v_m\|\ge |\langle v_n-v_m,e_n^*\rangle|=|\langle v_n,e_n^*\rangle|
=\alpha\Big(1-\Big(1-\frac1n\Big)^n\Big)\xrightarrow[n\to\infty]{}\alpha(1-e^{-1})>0.
$$
In particular, $(v_n)$ has no norm convergent subsequence, so $\widehat f$ is not compact
(by Proposition~\ref{mol}).

\smallskip
\noindent\emph{$\widehat f$ is weakly compact.}
For every $x\in B_{\ell_p}$ we have $|x_n^{\,n}|\le 1$, hence
$$
\|f(x)\|_p\le \|\alpha\lambda\|_p<1
\quad\text{and}\quad
f(B_{\ell_p})\subset \prod_{n\ge 2}\overline{D}\Big(0,\frac{\alpha}{n}\Big),
$$
which is a compact subset of $\ell_p$ (uniform tail control since $\left(\frac{\alpha}{n}\right)\in \ell_p$).
Therefore $\overline{f(B_{\ell_p})}$ is compact and contained in $B_{\ell_p}$.
Since $\ell_p$ is reflexive, Theorem~\ref{CSwcomp} yields that $\widehat f$ is weakly compact.
\end{example}

\subsection{The case when \texorpdfstring{$f$}{f} is linear}

\begin{cor}\label{cor:lin-compactness}
Assume that $f : X \to Y$ is linear and that $f(B_X) \subset B_Y$. Then, $\widehat{f} : \G_0(B_X) \to \G_0(B_Y)$ is compact if and only if $\overline{f(B_X)}$ is compact and  $\|f\|_{\infty}<1$.
\end{cor}

\begin{proof}
Assume that $\overline{f(B_X)}$ is compact and $\|f\|_{\infty}<1$. Then $\overline{f(B_X)}$ is contained in $B_Y$. Since $f$ is linear, $\{df(x)(w) : x \in B_X, w \in S_X\}=\{f(w) : w \in S_X\}$ is relatively compact by assumption. Hence, using Theorem~\ref{CScomp} and Remark~\ref{rmkDf}, $\widehat{f}$ is compact. 

Conversely, assume that $\widehat{f}$ is compact. It follows from Proposition~\ref{CNcomp} that $L_f$ is relatively compact, and then from Lemma~\ref{lemma:Lf-to-fBX} that $\overline{f(B_X)}$ is compact. By Theorem~\ref{CNcomp2}, \ref{p2} is satisfied and since $f$ is linear, Remark~\ref{remp4p1} gives $\|f\|_{\infty}<1$.
\end{proof}

\begin{corollary}
Let $f : X \longrightarrow \ell_p$ linear such that $f(B_X) \subset B_{\ell_p}$. Then:
$$\widehat{f} \text{ is compact } \iff \widehat{f} \text{ is weakly compact. }$$
\end{corollary}

\begin{proof}
Assume that $\widehat{f}$ is weakly compact. Then by Theorem \ref{recip}, $\overline{f(B_X)}$ is compact and $\|f\|_{\infty}<1$ and then by Corollary \ref{cor:lin-compactness}, $\widehat{f}$ is compact.  
\end{proof}

\subsection{The case when \texorpdfstring{$\dim(X)<\infty$}{dim(X)<infty}}

The following result has already been shown in \cite[Proposition~4.14]{LuisCarlos} with a different approach. 

\begin{corollary}
	Let $f \in \HL_0(B_X,Y)$  satisfy $f(B_X)\subset B_Y$ with $\dim(X)<\infty$ and $Y$ reflexive. Then:
	$$\widehat{f} \text{ is weakly compact } \iff \|f\|_{\infty}<1 \iff \overline{f(B_X)} \subset B_Y.$$
\end{corollary}

\begin{proof}
	The second equivalence is contained in Proposition~\ref{prop-images-relation}, and the first one is covered by Proposition~\ref{lem:wkcomp-strict-contraction} and Theorem~\ref{CSwcomp}.
\end{proof}

The next example sheds some light to the case when the domain $X$ is finite dimensional and the target space $Y$ is not reflexive. Note that, if $\dim(X)<\infty$, $\overline{B_X}$ is compact so that $\overline{f(B_X)}$ is automatically compact. 

\begin{example}
There exist Banach spaces $X,Y$ with $\dim(X)<\infty$ and $f \in \HL_0(B_X,Y)$ such that $\|f\|_{\infty}<1$ but $\widehat{f} : \G_0(B_X) \longrightarrow \G_0(B_Y)$ is not weakly compact. 

Take $X=\mathbb{R}$ and $Y=c_0$ the vector space of all the sequences which converge to~$0$, endowed with the supremum norm. Consider:
$$\fonction{f}{]-1;1[}{c_0}{x}{\left(\frac{1}{2n}x^n\right)_{n \geq 1}.}$$
It is clear that $f$ is well defined, holomorphic on $]-1;1[$, with $\|df\|_{\infty} \leq 1$ so that $f$ is Lipschitz with $\|f\|_L \leq 1$. It is also clear that $\|f\|_{\infty}<1$.

Consider the sequences $(x_k)$ and $(y_k)$ defined by $x_k=1$ and $y_k=1-\frac{1}{k}$ for $k \geq 1$. Then direct computations show that:
\begin{align*}\frac{f(x_k)-f(y_k)}{|x_k-y_k|}&=\frac{\left(\frac{1}{2n}\left(1-\left(1-\frac{1}{k}\right)^n\right)\right)_{n \geq 1}}{\frac{1}{k}}
=\left(  \frac{1}{2n} \sum\limits_{j=0}^{n-1} \left(1-\frac{1}{k}\right)^{j} \right)_{n \geq 1}.
\end{align*}
Each coordinate tends to $1$ when $k$ goes to $\infty$. Hence, this sequence cannot have any weakly convergent subsequence in $c_0$. Consequently, $L_f$ is not relatively weakly compact, and by Proposition \ref{CNcomp}, $\widehat{f}$ is not weakly compact.

\end{example}

Now, the question that naturally arises is: if $\dim(X)<\infty$ and $Y$ is reflexive, do we have the equivalence between compactness and weak compactness of $\widehat{f}$ ? The answer is no, as Example \ref{exdimfinie} will show. 
\smallskip

For this purpose, we need some tools about interpolating sequences in the unit disc of the complex plane $\mathbb{D}$, which will be helpful to construct a counterexample. The following definition and theorem, and further results on interpolating sequences can be found in John B. Garnett's book \cite{Garnett}.

In the sequel, we will denote $\mathbb{H}=\{z \in \mathbb{C} ~  | ~  \mathrm{Im}(z)>0\}$ the upper half-plane. The Banach space of all bounded analytic functions on $\mathbb{H}$ endowed with the supremum norm is denoted $H^{\infty}$. Similarly, $H^{\infty}(\mathbb{D})$ is the space of all bounded analytic functions on $\mathbb{D}$. 

\begin{definition}
    We say that a sequence $(z_k) \subset \mathbb{D}$ (resp. $(z_k) \subset \mathbb{H}$) is an interpolating sequence when the linear operator $$\fonction{R}{H^{\infty}(\mathbb{D})}{\ell_{\infty}}{f}{(f(z_j))_{j \geq 1}} \left(\text{ resp. } \fonction{R}{H^{\infty}}{\ell_{\infty}}{f}{(f(z_j))_{j \geq 1}}\right) $$ is surjective.
\end{definition}

Consider the Cayley transform $\varphi : \mathbb{H} \longrightarrow \mathbb{D}$ defined by $\varphi(z)=\frac{z-i}{z+i}$. It is a conformal mapping from $\mathbb{H}$ to $\mathbb{D}$, that is, a bijective biholomorphic map (see \cite{rudin} for background on conformal mappings). It is not hard to show that $(z_k)$ is an interpolating sequence in the upper half plane if and only if $(\varphi(z_k))$ is an interpolating sequence in $\mathbb{D}$. 
\smallskip

Let $(a_j) \in \ell_{\infty}$ and $(z_j)$ be an interpolating sequence in $\mathbb{D}$. Then, since the operator $R$ is surjective, the open mapping theorem gives a constant $M$ such that there exists $f \in H^{\infty}(\mathbb{D})$ with
$$\|f\|_{\infty} \leq M\sup_j |a_j|=M\|(a_j)\|_{\infty}.$$
The smallest such constant $M$ is called \emph{the constant of interpolation}:
$$M=\sup_{\|a_j\|_{\infty} \leq 1} \inf \{\|f\|_{\infty} : f \in H^{\infty}(\mathbb{D}), f(z_j)=a_j, ~   j \in \mathbb{N}\}.$$

\smallskip

The following theorem is stated in \cite[Theorem 2.1, p.285]{Garnett} with an interpolating sequence $(z_k) \subset \mathbb{H}$. It is rather straightforward to see that it is still true for an interpolating sequence in $\mathbb{D}$, using again the Cayley transform. It will be the key for constructing the counterexample. 

\begin{theorem}\label{interpolating}
  Let $(z_k)$ be an interpolating sequence in $\mathbb{D}$, and let
    $$M=\sup_{\|a_j\| \leq 1} \inf\{\|f\|_{\infty} : f \in H^{\infty}(\mathbb{D}), f(z_j)=a_j, j \geq 1\}$$
    be the constant of interpolation. Then there exist functions $f_j \in H^{\infty}(\mathbb{D})$ such that $f_j(z_j)=1$, $f_j(z_k)=0$ for $k \neq j$ and $\sum_j |f_j(z)|\leq M$.  
\end{theorem}

We can now give a counterexample of a function $f \in HL_0(B_X,Y)$ such that $f(B_X) \subset B_Y$, $\dim(X)<\infty$, $Y$ is reflexive, and $\widehat{f}$ is weakly compact but not compact:

\begin{example}\label{exdimfinie}
Let $(z_k) \subset \mathbb{D}$ be an interpolating sequence and $(f_j)$ the functions in $H^{\infty}(\mathbb{D})$ given by Theorem~\ref{interpolating}. Then, for all $z \in \mathbb{D}$, $\sum_j |f_j(z)| \leq M$.
Let
$$
\fonction{g}{\mathbb{D}}{\ell^2}{w}{\left(\frac{1}{2M}f_k(w)\right)_k.}
$$
Then, for every $w \in \mathbb{D}$,
$$
\|g(w)\|_{\ell_2} \leq \frac{1}{2M}\sum_k |f_k(w)| \leq \frac12.
$$
Consider $f$ the primitive of $g$ which vanishes at $0$, namely
$$
f(z)=\int_0^z g(w)\,dw,
$$
where the integral is understood in the Bochner sense. Then $f$ is holomorphic and Lipschitz since $f'=g$ is bounded. Moreover, $\|f\|_{\infty}<1$, so $\widehat{f}$ is weakly compact by Theorem~\ref{recip}. However,
$$
f'(z_k)=g(z_k)=\frac{e_k}{2M},
$$
so $D_f$ is not relatively compact. Hence, $\widehat{f}$ is not compact by Remark~\ref{rmkDf} and Proposition~\ref{CNcomp}.
\end{example}

\section*{Acknowledgments}
The authors are grateful to L. C. García-Lirola and A. Proch\'azka for sharing with us several preliminary versions of their paper \cite{LuisCarlos}. They also wish to thank Mikhail Mironov for providing Example~\ref{exdimfinie}. The first and last authors were partially supported by the French ANR project No.~ANR-24-CE40-0892.

\section*{AI STATEMENT}

The authors of this article declare that none of them used any AI to find the results presented.

\end{document}